%% file: bose.tex
\documentclass[12pt]{article}
\usepackage{amsmath,amssymb,amsfonts,amsthm,tikz}

\usepackage[margin=1in]{geometry}
\usepackage{verbatim}

\usepackage{graphicx}
\usepackage{subcaption}

\theoremstyle{plain}

\newtheorem{theorem}{Theorem}[section]
\newtheorem{proposition}[theorem]{Proposition}
\newtheorem{construction}[theorem]{Construction}
\newtheorem{corollary}[theorem]{Corollary}
\newtheorem{lemma}[theorem]{Lemma}

\theoremstyle{remark}

\newtheorem*{remark}{Remark}

\newcommand{\Z}{\mathbb{Z}}

\input{letter-defn.txt}

\title{Simultaneous current graph constructions for minimum triangulations and complete graph embeddings}
\author{Timothy Sun\\Columbia University}
\date{}

\begin{document}

\maketitle

\begin{abstract}
The problems of the genus of the complete graphs and minimum triangulations for each surface were both solved using the theory of current graphs, and each of them divided into twelve different cases, depending on the residue modulo 12 of the number of vertices. Cases 8 and 11 were of particular difficulty for both problems, with multiple families of current graphs developed to solve these cases. We solve these cases in a unified manner with families of current graphs applicable to both problems. Additionally, we give new constructions to both problems for Cases 6 and 9, which greatly simplify previous constructions by Ringel, Youngs, Guy, and Jungerman. All these new constructions are index 3 current graphs sharing nearly all of the structure of the simple solution for Case 5 of the Map Color Theorem. 
\end{abstract}

\section{Introduction}

In this paper, we only consider surfaces which are orientable. We let $S_g$ denote the surface of genus $g$, i.e., the sphere with $g$ handles. The \emph{Heawood number} of the orientable surface $S_g$ of genus $g$,
$$H(S_g) = \frac{7+\sqrt{1+48g}}{2}$$
gives rise to two distinct problems which share many similarities. On one hand, the Heawood number is an \emph{upper bound} on the chromatic number of the surface, and the celebrated Map Color Theorem of Ringel, Youngs, and others~\cite{Ringel-MapColor} proves that this inequality is tight for all surfaces of genus $g \geq 1$ by determining the genus of the complete graphs. In the reverse direction, $H(S_g)$ is a \emph{lower bound} on the minimum number of vertices needed to triangulate the surface with a simple graph. For $g \geq 1$, $g \neq 2$, this was also shown to be tight by Jungerman and Ringel~\cite{JungermanRingel-Minimal}.

Both of these problems break down into twelve cases, where ``Case $k$'' refers to the relevant graphs on $12{s}+k$ vertices. The main tool for constructing most of the required embeddings is the theory of current graphs~\cite{Gustin}. At times, there is overlap---for example, the complete graph $K_7$ triangulates the torus, thereby demonstrating that the chromatic number of the torus and the smallest number of vertices needed to triangulate the torus is 7. However, many of the cases are solved separately, and furthermore, the latter problem of \emph{minimum triangulations}\footnote{Jungerman and Ringel~\cite{JungermanRingel-Minimal} used the less accurate term \emph{minimal triangulations}.} often required multiple unrelated families of current graphs. 

Our goal is a partial unification of both problems using \emph{index 3} current graphs, i.e., those which are embedded with three faces. The standard solutions for Cases 3 and 5 of the Map Color Theorem, i.e., the genus of the complete graphs on $12s{+}3$ and $12s{+}5$ vertices, respectively, used simple families of index 3 current graphs whose origins can be traced back to constructions for Steiner triple systems. However, other constructions employing index 3 current graphs, perhaps most notably Case 6 of the Map Color Theorem (see \S9.3 of Ringel~\cite{Ringel-MapColor}), have not realized the same level of simplicity. For each of Cases 6, 8, 9, and 11, we present a single family of current graphs which solves both the complete graph and minimum triangulation problems except for a few small-order graphs or surfaces. Not only do these constructions improve upon past solutions in the literature, but the structure of the current graphs for the general case reuses all but a finite part of the aforementioned current graphs used for Case 5.

\section{Embeddings in surfaces and the Heawood numbers}

For background in topological graph theory, see Gross and Tucker~\cite{GrossTucker}. In a graph, possibly with self-loops or parallel edges, every edge has two \emph{ends} that are each incident with a vertex. A \emph{rotation} of a vertex is a cyclic permutation of its incident edge ends, and a \emph{rotation system} of a graph is an assignment of a rotation to every vertex of the graph. The Heffter-Edmonds principle states that cellular embeddings of a graph are in one-to-one correspondence with rotation systems: each embedding in a surface defines a rotation system by considering the cyclic order of the edge ends emanating at each vertex, while in the reverse direction, the faces of the embedding can be traced out from the rotation system in a unique manner. Our convention will be that rotations define \emph{clockwise} orderings, which induce \emph{counterclockwise} orientations for faces. In the case of simple graphs, one can express a rotation in terms of the vertex's neighbors, so a rotation system can be represented as a table of vertices, where each row is a cyclic permutation of vertices. 

The \emph{Euler polyhedral formula} states that for a cellular embedding $\phi: G \to S_g$, we have the expression
$$|V(G)|-|E(G)|+|F(G,\phi)|=2-2g,$$
where $g$ denotes the genus of the surface and $F(G,\phi)$ is the set of faces induced by the embedding. A standard consequence is the following inequality:
\begin{proposition}
For a simple graph $G$ which embeds in $S_g$, 
$$|E(G)| \leq 3|V(G)|-6+6g,$$
where equality is achieved when the embedding is \emph{triangular}, i.e. when all its faces are triangular.
\label{prop-triangle}
\end{proposition} 

The \emph{genus} of a graph $G$ is the minimum genus over all cellular embeddings of $G$, and is denoted $\gamma(G)$. A \emph{genus embedding} of $G$ is an embedding whose genus achieves this minimum.

\begin{corollary}
For a simple graph $G$, its genus is at least
$$\gamma(G) \geq \left\lceil\frac{|E(G)|-3|V(G)|+6}{6}\right\rceil.$$
\label{cor-lowerbound}
\end{corollary}
From these relationships between the edge and vertex counts and the genus, one can derive the \emph{Heawood number}
$$H(g) = \frac{7+\sqrt{1+48g}}{2}$$
of the surface $S_g$, which serves as a rough measure of ``maximum possible density'' in the following two inequalities:

\begin{proposition}[{see Ringel~\cite[p.63]{Ringel-MapColor}}]
For $g \geq 1$, the chromatic number $\chi(S_g)$ of the surface $S_g$, i.e., the maximum chromatic number over all graphs embeddable in $S_g$, satisfies $$\chi(S_g) \leq H(S_g).$$
\label{prop-ry}
\end{proposition}
\begin{proposition}[{Jungerman and Ringel~\cite{JungermanRingel-Minimal}}]
Let $MT(S_g)$ be the minimum number of vertices over all simple graphs $G$ that have a triangular embedding in $S_g$. Then $$MT(S_g) \geq H(S_g).$$ 
\label{prop-jr}
\end{proposition}

Such an embedding in Proposition~\ref{prop-jr} is known as a \emph{minimum triangulation} of $S_g$. 
We call a triangular embedding of a graph an \emph{(n,t)-triangulation} if the graph has $n$ vertices and $\binom{n}{2}-t$ edges, i.e. the graph is the complete graph on $n$ vertices with $t$ edges deleted. The tightness of the inequalities in Propositions~\ref{prop-ry} and \ref{prop-jr} is proven via alternative formulations that emphasize the number of vertices:

\begin{theorem}[Ringel and Youngs~\cite{RingelYoungs}]
The genus of the complete graph $K_n$ is $$\gamma(K_n) = \left\lceil \frac{(n-3)(n-4)}{12} \right\rceil.$$
\end{theorem}
\begin{theorem}[Jungerman and Ringel~\cite{JungermanRingel-Minimal}]
For all pairs of integers $(n,t) \neq (9,3)$, where 
\begin{align*}
&n \geq 4, 0 \leq t \leq n-6,\\
&(n-3)(n-4) \equiv 2t \pmod{12}, 
\end{align*}
there exists an $(n,t)$-triangulation.
\label{thm-mt}
\end{theorem}

In both problems, the proof breaks down into several cases, depending on the residue of the number of vertices $n \bmod{12}$. We call the subcase concerning graphs with $n = 12s{+}k$ vertices \emph{Case $k$}, for $k = 0, 1, \dotsc, 11$, and we often reference the value $s$ in our exposition. For example, if we speak of ``Case 6, $s = 2$'' of the Map Color Theorem, we are referring to the complete graph $K_{30}$. To differentiate between the two problems, we refer to ``Case $k$-CG'' and ``Case $k$-MT'' to denote Case $k$ of the Map Color Theorem (``complete graph'') and minimum triangulations problem, respectively. 

The fact that there are 12 Cases depending on the number of vertices for both the Map Color Theorem and the minimum triangulations problem suggests that there might be a connection between the solutions of the two problems. Indeed, in several Cases, the current graphs used in the proof of the Map Color Theorem~\cite{Ringel-MapColor} for $K_n$ have the dual purpose of also providing all the necessary minimum triangulations on the same number of vertices $n$. However, not all Cases have been combined in this manner. 

In general, our constructions will proceed in the following way: using an index 3 current graph, we generate an $(n,t)$-triangulation. We wish to find other embeddings of graphs on the same number of vertices using the following operations:

\begin{itemize}
\item \emph{Handle subtraction}, which deletes edges from a triangular embedding to produce a triangular embedding on a lower-genus surface. 
\item \emph{Additional adjacency}, which adds edges using extra handles and edge flips.
\end{itemize} 

By subtracting handles, we obtain all the necessary $(n,t')$-triangulations, for $t' > t$, and over the course of the additional adjacency step for constructing a genus embedding of $K_n$, we construct the remaining $(n,t'')$-triangulations, for $t'' < t$. We note that the genus of the complete graphs achieves the lower bound in Corollary~\ref{cor-lowerbound}.

\section{Outline for additional adjacencies}

The main goal for our additional adjacency steps is to utilize as little information about the embeddings as possible. For this reason, we present the additional adjacency solutions first, before describing any current graphs. Like in previous work, our additional adjacency solutions make use of three different operations for adding a handle, which are described in Constructions~\ref{prop-3handle}, \ref{prop-4handle}, and \ref{cons-mergehandle} in primal form. In prose, we describe the modifications to the embeddings in terms of rotation systems, so their correctness can be checked by tracing the faces and applying the Heffter-Edmonds principle. Our drawings, on the other hand, describe an alternate topological interpretation using surgery on the embedded surfaces. While these operations work more generally, we assume that all graphs in this section are simple.

\begin{construction}
Modifying the rotation at vertex $v$ from
$$\begin{array}{rrrrrrrrrrrrrrrrrrrrrrrrrrrrrrrr}
v. & x_1 & \dots & x_i & y_1 & \dots & y_j & z_1 & \dots & z_k
\end{array}$$
to
$$\begin{array}{rrrrrrrrrrrrrrrrrrrrrrrrrrrrrrrr}
v. & x_1 & \dots & x_i & z_1 & \dots & z_k & y_1 & \dots & y_j,
\end{array}$$
as in Figure~\ref{fig-3handle} increases the genus by 1 and induces the 9-sided face
$$[x_1, z_k, v, y_1, x_i, v, z_1, y_j, v]$$
\label{prop-3handle}
\end{construction}

\begin{construction}
Modifying the rotation at vertex $v$ from
$$\begin{array}{rrrrrrrrrrrrrrrrrrrrrrrrrrrrrrrr}
v. & x_1 & \dots & x_i & y_1 & \dots & y_j & z_1 & \dots & z_k & w_1 & \dots & w_l
\end{array}$$
to
$$\begin{array}{rrrrrrrrrrrrrrrrrrrrrrrrrrrrrrrr}
v. & x_1 & \dots & x_i & w_1 & \dots & w_l & z_1 & \dots & z_k & y_1 & \dots & y_j
\end{array}$$
as in Figure~\ref{fig-4handle} increases the genus by 1 and induces the two 6-sided faces
$$[x_1, w_l, v, z_1, y_j, v]\textrm{\hspace{0.25cm}and\hspace{0.25cm}}[w_1, z_k, v, y_1, x_i, v].$$
\label{prop-4handle}
\end{construction}

\begin{figure}[!ht]
    \begin{subfigure}[b]{0.7\textwidth}
        \centering
        \includegraphics[scale=0.9]{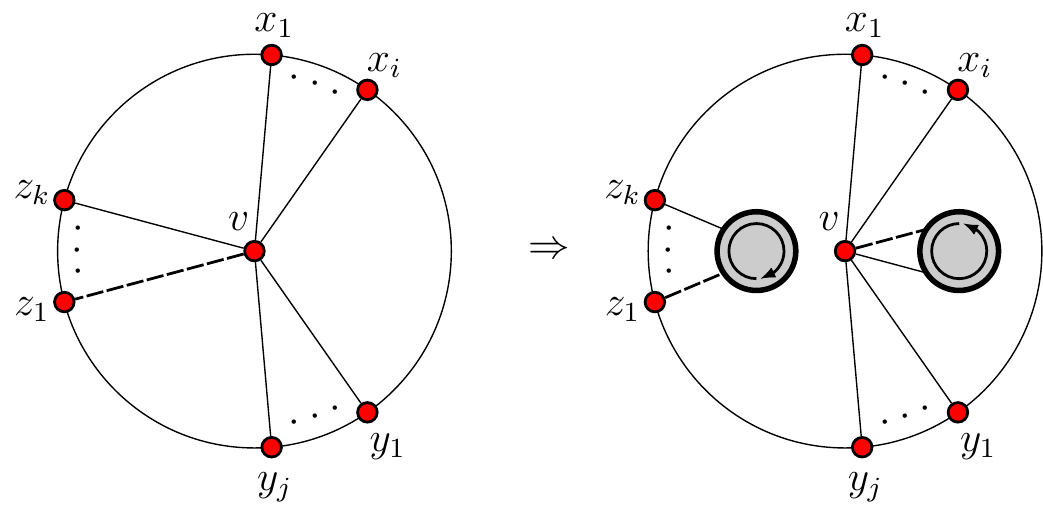}
        \caption{}
        \label{subfig-heff3-a}
    \end{subfigure}
    \begin{subfigure}[b]{0.28\textwidth}
        \centering
        \includegraphics[scale=0.9]{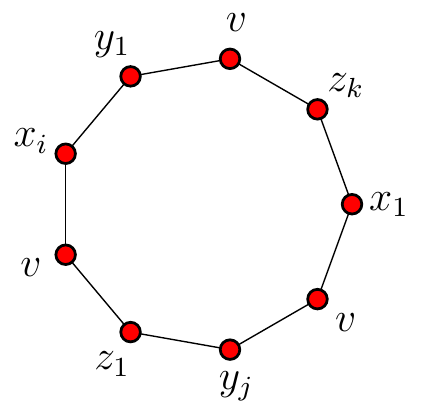}
        \caption{}
        \label{subfig-heff3-b}
    \end{subfigure}
\caption{Rearranging the rotation at vertex $v$ (a) increases the genus and creates room (b) to add new edges.}
\label{fig-3handle}
\end{figure}

\begin{figure}[!ht]
    \begin{subfigure}[b]{\textwidth}
        \centering
        \includegraphics[scale=0.9]{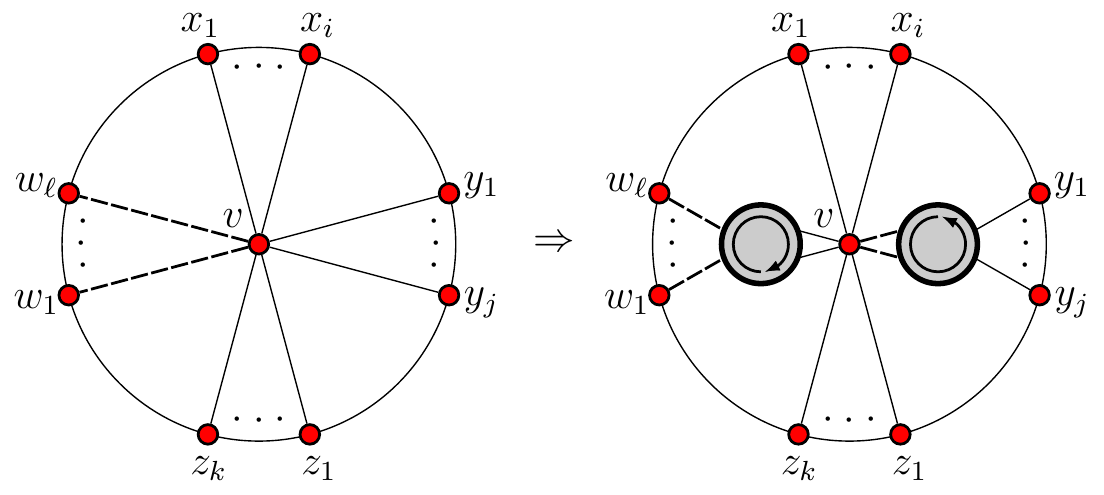}
        \caption{}
        \label{subfig-heff4-a}
    \end{subfigure}
    \begin{subfigure}[b]{\textwidth}
        \centering
        \includegraphics[scale=0.9]{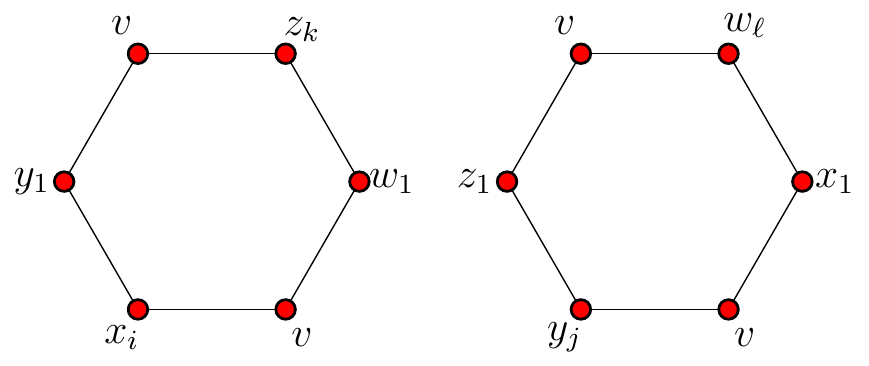}
        \caption{}
        \label{subfig-heff4-b}
    \end{subfigure}
\caption{Rearranging four groups of neighbors (a) yields two hexagonal faces (b).}
\label{fig-4handle}
\end{figure}

\begin{remark}
While the drawings in Figures \ref{fig-3handle} and \ref{fig-4handle} are drawn asymmetrically, the operations are in fact invariant under cyclic shifts of the subsets $x_1, \dotsc, x_i$; $y_1, \dotsc, y_j$, etc. 
\end{remark}

Several Cases of the Map Color Theorem are solved by first finding triangular embeddings of $K_n-K_3$. The first consequence of Construction~\ref{prop-3handle} is to transform such an embedding into a genus embedding of a complete graph.

\begin{proposition}[Ringel~\cite{Ringel-1961}]
If there exists a triangular embedding $\phi: K_n{-}K_3 \to S_g$, then there exist a genus embedding of $K_n$ in the surface $S_{g+1}$.
\label{cor-k3}
\end{proposition}

Before showing how this follows from the above constructions, we first argue that all the embeddings of complete graphs we construct are in fact of minimum genus.

\begin{proposition}
Suppose we have a triangular embedding of a graph $K_n-H_e$, where $H_e$ is a graph on $e$ edges, $e < 6$. If we add the missing $e$ edges by using one handle, the resulting embedding is a genus embedding of $K_n$.
\label{prop-tight}
\end{proposition}
\begin{proof}
One can verify that the difference between the genus of $K_n-H_e$, as given by Proposition~\ref{prop-triangle}, and the genus of $K_n$ is exactly 1.
\end{proof}

\begin{proof}[Proof of Proposition~\ref{cor-k3}]
If the three nonadjacent vertices are $\vora,\vorb,\vorc$, pick any other vertex $v$ and apply Construction~\ref{prop-3handle} with $x_1 = \vora, y_1 = \vorb, z_1 = \vorc$. In the resulting nontriangular face, the nonadjacent vertices can be connected like in Figure~\ref{fig-add-k14}(a). 
\end{proof}

For Cases 8 and 11, we will construct triangular embeddings of the graph $K_n-K_{1,4}$. These missing edges can be added in using one handle if the embedding satisfies an additional constraint:

\begin{proposition}
Let $K_n-K_{1,4}$ be a complete graph with the edges $(u, q_1), \dotsc, (u, q_4)$ deleted. If there exists a triangular embedding $\phi: (K_n-K_{1,4}) \to S_g$ with a vertex $v$ with rotation
$$\begin{array}{rrrrrrrrrrrrrrrrrrrrrrrrrrrrrrrr}
v. & \dots & q_1 & q_2 & \dots & q_3 & q_4 & \dots,
\end{array}$$
then there exists a genus embedding of $K_n$ in the surface $S_{g+1}$.
\label{prop-add-k14}
\end{proposition}
\begin{proof}
Note that vertices $u$ and $v$ are adjacent, so assume without loss of generality that $u$ appears in the rotation of $v$ in between $q_4$ and $q_1$. Apply Construction~\ref{prop-3handle} with $$x_i = q_1, y_1 = q_2, y_j = q_3, z_1 = q_4, z_k = u$$ and connect the missing edges in the 9-sided face, as in Figure~\ref{fig-add-k14}(b). 

\begin{figure}[!ht]
    \centering
    \begin{subfigure}[b]{0.4\textwidth}
        \centering
        \includegraphics[scale=1.0]{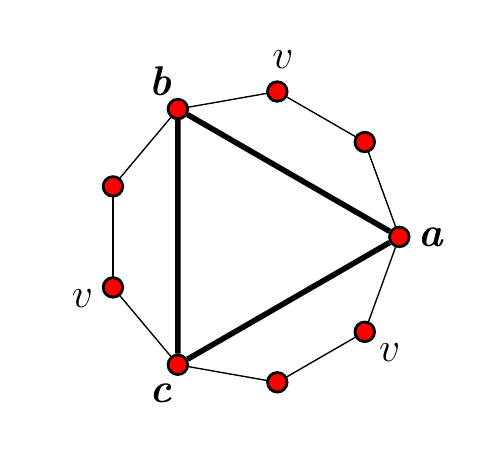}
        \caption{}
        \label{subfig-k14-a}
    \end{subfigure}
    \begin{subfigure}[b]{0.4\textwidth}
        \centering
        \includegraphics[scale=1.0]{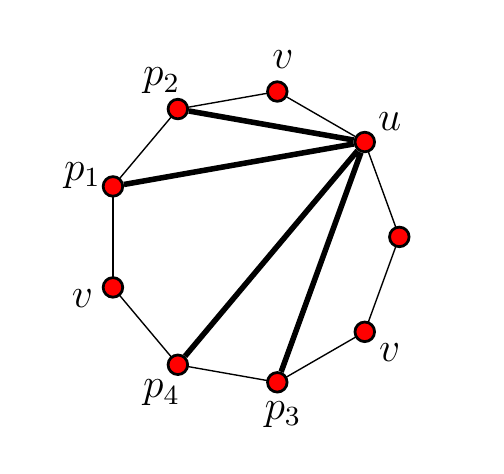}
        \caption{}
        \label{subfig-k14-b}
    \end{subfigure}
\caption{Two possibilities for adding edges after invoking Construction~\ref{prop-3handle}: a $K_3$ subgraph (a), and a $K_{1,4}$ subgraph (b).}
\label{fig-add-k14}
\end{figure}
\end{proof}

This constraint is relatively easy to satisfy, since there are a few possible permutations for $q_1, \dotsc, q_4$, in addition to the fact that $v$ is an arbitrary vertex. In fact, when we only need to add back three edges, this is always possible:

\begin{corollary}[Ringel \emph{et al.}~\cite{RingelYoungs-Case2, GuyRingel}]
If there exists a triangular embedding $\phi: K_n{-}K_{1,3} \to S_g$, then there exist a genus embedding of $K_n$ in the surface $S_{g+1}$.
\label{cor-k13}
\end{corollary}
\begin{proof}
One can always find such a vertex $v$ by choosing a vertex on one of the triangles incident with, say, the edge $(q_1, q_2)$.
\end{proof}

A third type of handle operation is to merge two faces with a handle without modifying the rotations at any vertices. To do this, we excise a disk from two faces and identify the resulting boundaries. In Figure~\ref{fig-polyhandle}, adding the handle between faces $F_1$ and $F_2$ causes the embedding to become noncellular, as the resulting region is an annulus. However, once we start adding edges between the two boundary components of the annulus, the embedding becomes cellular again. 

\begin{construction}
Let $F_1 = [u_1, u_2, \dotsc, u_i]$ and $F_2 = [v_1, v_2, \dotsc, v_j]$ be two faces. Inserting the edge $(u_1, v_1)$ in the following way
$$\begin{array}{rrrrrrrrrrrrrrrrrrrrrrrrrrrrrrrr}
u_1. & \dots & u_i & u_2 & \dots \\
v_1. & \dots & v_j & v_2 & \dots \\
\end{array} 
\Rightarrow
\begin{array}{rrrrrrrrrrrrrrrrrrrrrrrrrrrrrrrr}
u_1. & \dots & u_i & v_1 & u_2 & \dots \\
v_1. & \dots & v_j & u_1 & v_2 & \dots \\
\end{array}$$
as in Figure~\ref{fig-polyhandle} increases the genus by 1 and induces the $(i+j+2)$-sided face
$$[u_1, u_2, \dotsc, u_i, u_1, v_1, v_2, \dotsc, v_j, v_1].$$
\label{cons-mergehandle}
\end{construction}

\begin{figure}[ht]
\centering
\includegraphics[scale=0.9]{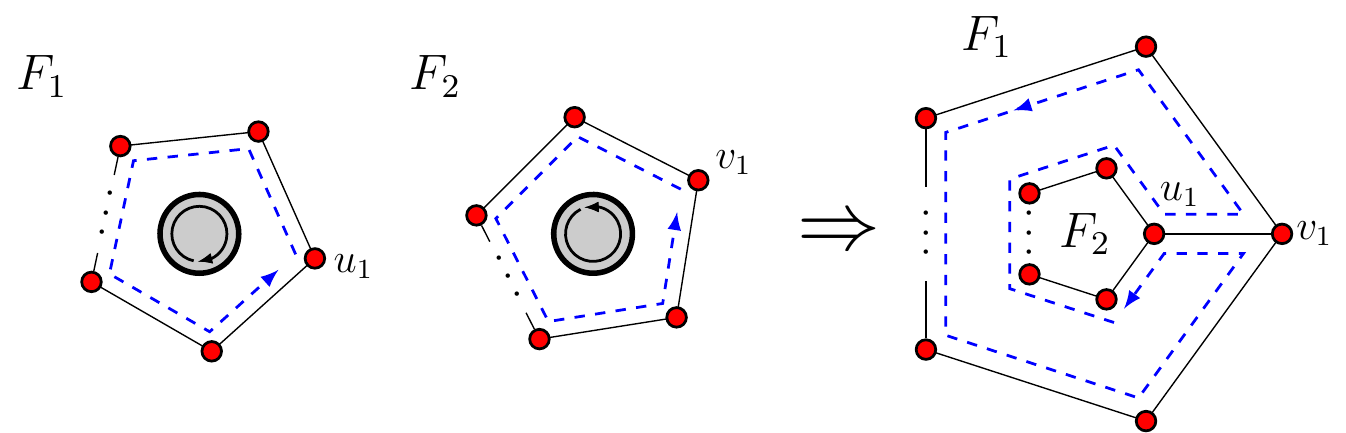}
\caption{Adding a handle between two faces, then adding an edge to transform the annulus into a cell. Note that the order of vertices of one of the faces becomes reversed as we traverse one of the (oriented) boundaries the annulus.}
\label{fig-polyhandle}
\end{figure}

The most elementary operation one can do is to simply add one edge to create a genus embedding:

\begin{proposition}
If there exists a triangular embedding $\phi: K_n{-}K_2 \to S_g$, then there exist a genus embedding of $K_n$ in the surface $S_{g+1}$.
\label{prop-add-k2}
\end{proposition}
\begin{proof}
Let $v_1$ and $v_2$ be the two nonadjacent vertices. Let $F_i$ be a face incident with $v_i$ for $i = 1,2$ and apply Construction~\ref{cons-mergehandle}.
\end{proof}

The forthcoming additional adjacency solutions are to be applied on triangular embeddings of graphs of the form $K_n-K_\ell$, which is the graph formed by taking the complete graph $K_n$ and deleting all the pairwise adjacencies between $\ell$ vertices. We label the vertices missing adjacencies with bold letters $\vora, \vorb, \vorc, \dotsc, \vorh$. The remaining vertices will be assigned numbers and are represented here as unadorned letters ($u, v, p_i, \dotsc$). We apply the traditional method of adding handles to supply all the missing edges---in Section~\ref{sec-alternate}, we give an alternative viewpoint that aims to demystify the specific choices of added edges. 

\begin{lemma}
If there exists a triangular embedding of $K_n{-}K_5$ with numbered vertices $u$ and $v$ whose rotations are of the form
$$\begin{array}{rrrrrrrrrrrrrrrrrrrrrrrrrrrrrrrr}
u. & \dots & \textbf{a} & p_1 & \textbf{b} & p_2 & \textbf{c} & p_3 & \textbf{d} & p_4 & \textbf{e} & \dots
\end{array}$$
and
$$\begin{array}{rrrrrrrrrrrrrrrrrrrrrrrrrrrrrrrr}
v. & \dots & p_{\sigma(1)} & p_{\sigma(2)} & \dots & p_{\sigma(3)} & p_{\sigma(4)} &\dots,
\end{array}$$
where $\sigma: \{1,\dotsc,4\}\to\{1,\dotsc,4\}$ is some permutation, then there exist $(n,10)$- and $(n,4)$-triangulations and a genus embedding of $K_n$.
\label{lem-k5}
\end{lemma}
\begin{proof}
The initial embedding is an $(n,10)$-triangulation. First, delete the edges $(u, p_1)$, $(u, \vorb)$, $(u, p_2)$ in exchange for $(\vora, \vorb)$, $(\vora, \vorc)$, $(\vorb,\vorc)$ and apply edge flips on $(u, p_3)$ and $(u, p_4)$ to obtain $(\vorc, \vord)$ and $(\vord, \vore)$, as in Figure~\ref{fig-k5-add}(a). If we merge the faces $[\vora, \vorc, \vorb]$ and $[u, \vore, \vord]$ with a handle, we can recover the deleted edge $(u, \vorb)$ and add in the remaining edges between lettered vertices following Figure~\ref{fig-k5-add}(b). The missing edges $(u, p_1), \dotsc, (u, p_4)$ in this $(n,4)$-triangulation can be reinserted with one handle using Proposition~\ref{prop-add-k14}, setting $p_{\sigma(i)} = q_i$, to get a genus embedding of $K_n$. 

\begin{figure}[!ht]
    \centering
    \begin{subfigure}[b]{0.69\textwidth}
        \centering
        \includegraphics[scale=1.0]{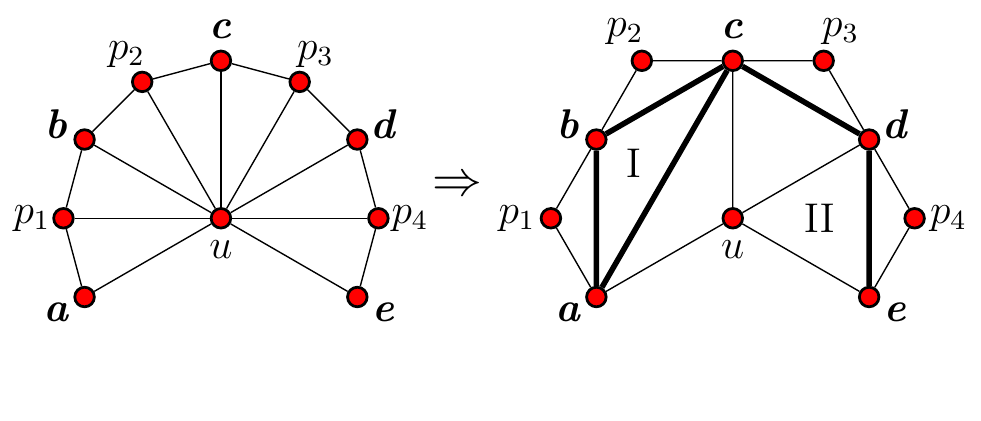}
        \caption{}
        \label{subfig-k5-a}
    \end{subfigure}
    \begin{subfigure}[b]{0.29\textwidth}
        \centering
        \includegraphics[scale=1.0]{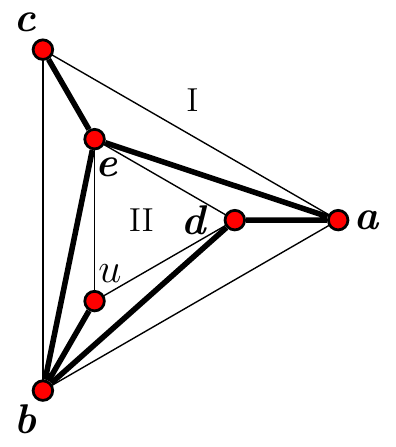}
        \caption{}
        \label{subfig-k5-b}
    \end{subfigure}
\caption{Various edge flips are applied in the neighborhood of vertex $u$ (a) so that one handle suffices for connecting all the lettered vertices.}
\label{fig-k5-add}
\end{figure}
\end{proof}

\begin{lemma}[Guy and Ringel~\cite{GuyRingel}]
If there exists a triangular embedding of $K_n{-}K_6$ with a numbered vertex $u$ whose rotations are of the form
$$\begin{array}{rrrrrrrrrrrrrrrrrrrrrrrrrrrrrrrr}
u. & \dots & \textbf{a} & p_1 & \textbf{b} & \dots & \textbf{c} & p_2 & \textbf{d} & \dots & \textbf{e} & p_3 & \textbf{f} & \dots,
\end{array}$$
then there exist $(n,15)$-, $(n,9)$-, and $(n,3)$-triangulations and a genus embedding of $K_n$.
\label{lem-k6}
\end{lemma}
\begin{proof}
We first modify the embedding near vertex $u$ using edge flips to gain the edges $(\vora, \vorb)$, $(\vorc, \vord)$, and $(\vore, \vorf)$, as in Figure~\ref{fig-k6-add}(a). If we apply Construction~\ref{prop-3handle} to vertex $u$, we obtain a 9-sided face incident with all six vertices $\vora, \vorb, \dotsc, \vorf$. In Figure~\ref{fig-k6-add}(b) and (c), we give one way to insert the nine missing edges between these lettered vertices with the help of a handle.

\begin{figure}[!ht]
    \centering
    \begin{subfigure}[b]{0.80\textwidth}
        \centering
        \includegraphics[scale=1.0]{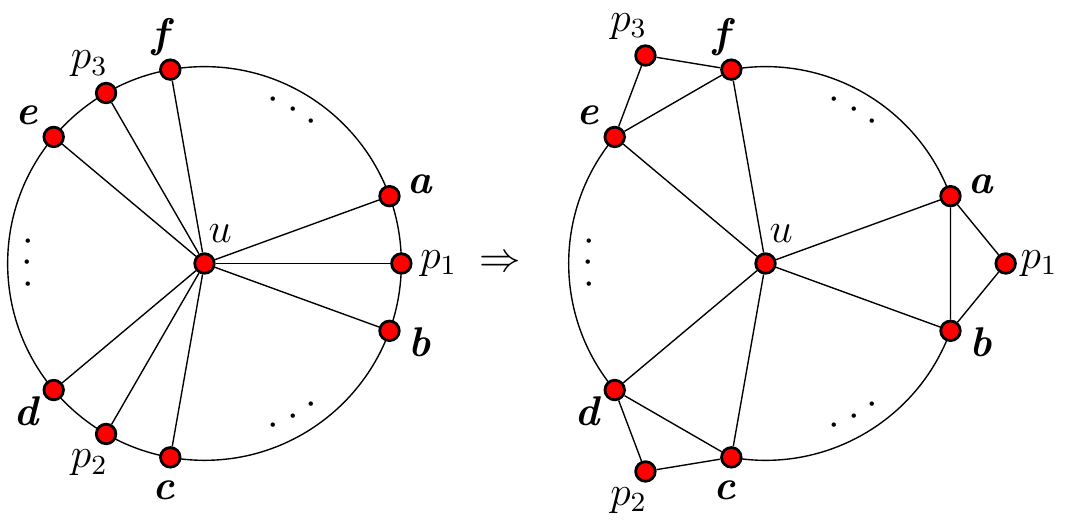}
        \caption{}
        \label{subfig-k6-a}
    \end{subfigure}
    \begin{subfigure}[b]{0.35\textwidth}
        \centering
        \includegraphics[scale=1.0]{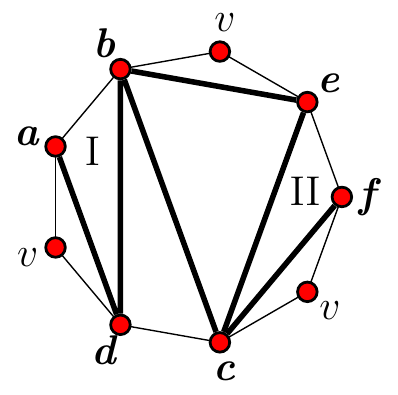}
        \caption{}
        \label{subfig-k6-b}
    \end{subfigure}
    \begin{subfigure}[b]{0.35\textwidth}
        \centering
        \includegraphics[scale=1.0]{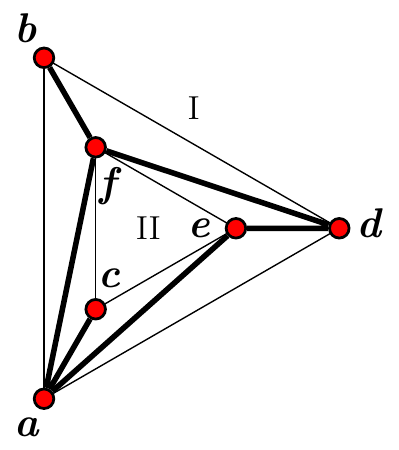}
        \caption{}
        \label{subfig-k6-c}
    \end{subfigure}
\caption{Three pairs of lettered vertices are connected with some edge flips (a), after which a handle adds some of the missing adjacencies (b). The remaining edges between lettered vertices are added using another handle merging faces I and II (c).}
\label{fig-k6-add}
\end{figure}

The missing edges $(u,p_1), (u,p_2), (u,p_3)$ can be added back using Corollary~\ref{cor-k13}, yielding a genus embedding of $K_n$.
\end{proof}

\begin{lemma}
If there exists a triangular embedding of $K_n{-}K_8$ with numbered vertices $u$ and $v$ whose rotations are of the form
$$\begin{array}{rrrrrrrrrrrrrrrrrrrrrrrrrrrrrrrr}
u. & \dots & \leta & p_1 & \letb & \dots & \letc & p_2 & \letd & \dots & \lete & p_3 & \letf & \dots & \letg & p_4 & \leth & \dots
\end{array}$$
and
$$\begin{array}{rrrrrrrrrrrrrrrrrrrrrrrrrrrrrrrr}
v. & \dots & p_{\sigma(1)} & p_{\sigma(2)} & \dots & p_{\sigma(3)} & p_{\sigma(4)} &\dots,
\end{array}$$
where $\sigma: \{1,\dotsc,4\}\to\{1,\dotsc,4\}$ is some permutation, then there exist $(n,28)$-, $(n,22)$-, $(n,16)$-, $(n,10)$-, and $(n,4)$-triangulations and a genus embedding of $K_n$.
\label{lem-k8}
\end{lemma}
\begin{proof}
The first four handles of our additional adjacency approach is the same as that of Ringel and Youngs' solution for Case 2-CG~\cite{RingelYoungs-Case2} (also see Ringel~\cite[\S7.5]{Ringel-MapColor}), with different vertex names. We perform an edge flip on each edge $(u, p_i)$ for $i = 1, \dotsc, 4$, gaining the edges $(\vora,\vorb)$, $(\vorc,\vord)$, $(\vore,\vorf)$, and $(\vorg,\vorh)$. Now, the rotation at vertex $u$ is of the form
$$\begin{array}{rrrrrrrrrrrrrrrrrrrrrrrrrrrrrrr}
u. & \dots & \vora & \vorb & \dots & \vorc & \vord & \dots & \vore & \vorf & \dots & \vorg & \vorh & \dots
\end{array}$$
These edge flips are depicted in Figure~\ref{fig-vort8}. Applying Construction~\ref{prop-4handle} to this resulting rotation yields two nontriangular faces
$$[\vorh, \vorg, v, \vord, \vorc, v]\textrm{\hspace{0.25cm}and\hspace{0.25cm}}[\vorf, \vore, v, \vorb, \vora, v].$$

\begin{figure}[ht]
\centering
\includegraphics{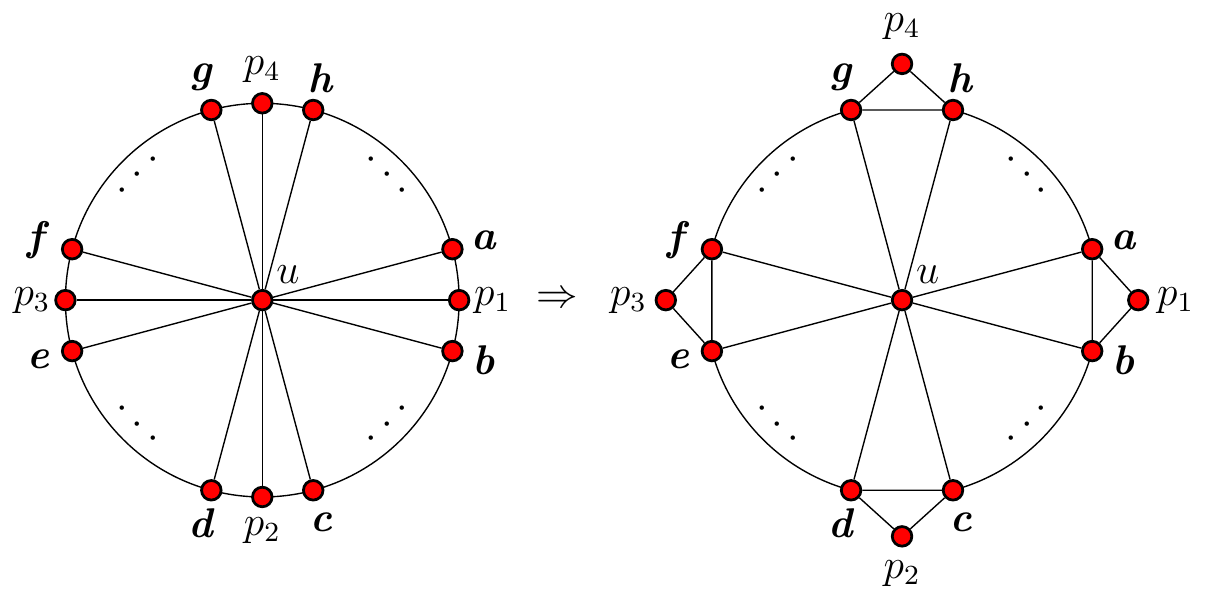}
\caption{Initial edge flips to join some of the vortex letters.}
\label{fig-vort8}
\end{figure}

In these faces, we induce two quadrilateral faces by adding the edges $(\vord, \vorg)$, $(\vorc, \vorh)$, $(\vorb, \vore)$, and $(\vora, \vorf)$, as in Figure~\ref{fig-k8-add}(a). Three more handles are used to add all the remaining edges between lettered vertices $\vora, \dotsc, \vorh$ as shown in Figure~\ref{fig-k8-add}(bc). At this point, the embedding is of the graph $K_n-K_{1,4}$ and is still triangular, so we replace the deleted edges $(u, p_i)$ with one handle using Proposition~\ref{prop-add-k14} to obtain a genus embedding of $K_n$. 

\begin{figure}[!ht]
    \centering
    \begin{subfigure}[b]{0.6\textwidth}
        \centering
        \includegraphics[scale=1.0]{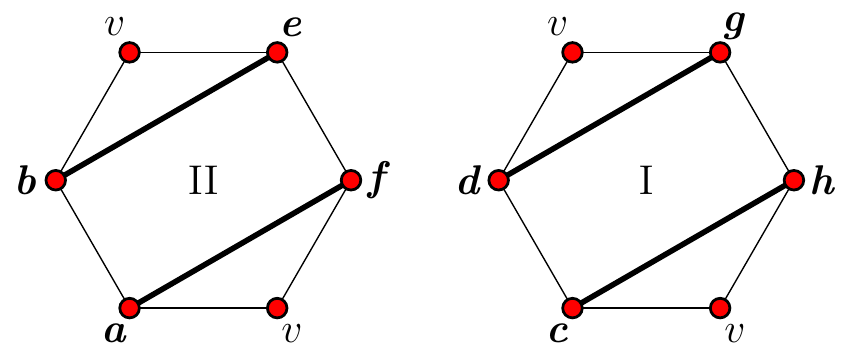}
        \caption{}
        \label{subfig-k8-a}
    \end{subfigure}
    \begin{subfigure}[b]{0.3\textwidth}
        \centering
        \includegraphics[scale=1.0]{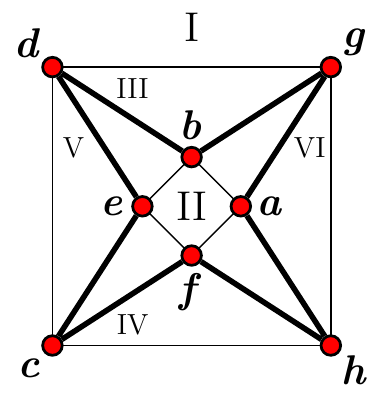}
        \caption{}
        \label{subfig-k8-b}
    \end{subfigure}
    \begin{subfigure}[b]{0.5\textwidth}
        \centering
        \includegraphics[scale=1.0]{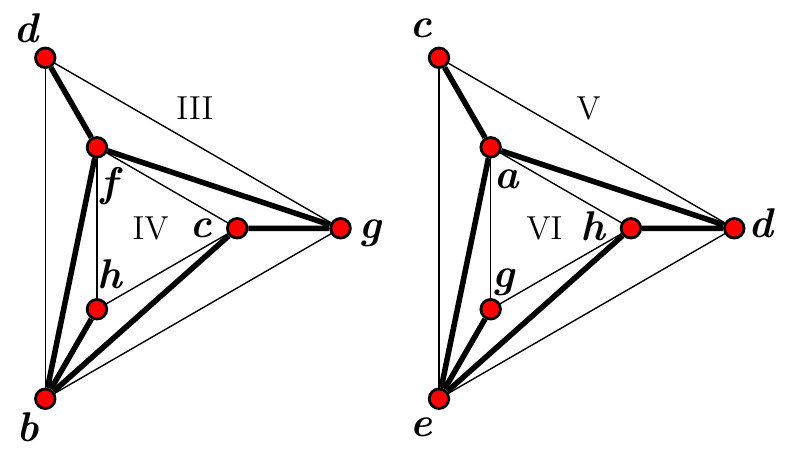}
        \caption{}
        \label{subfig-k8-c}
    \end{subfigure}
\caption{After connecting some of the lettered vertices with a handle (a), another handle can be introduced in between the faces I and II (b). Using faces generated from this handle (III and IV, V and VI), we can add all the remaining edges using two additional handles (c).}
\label{fig-k8-add}
\end{figure}

The embeddings after adding the second through fourth handles are all triangular and hence are minimum triangulations. After adding only the first handle, the two quadrilateral faces in Figure~\ref{fig-k8-add}(a) can be triangulated arbitrarily to form an $(n,22)$-triangulation.
\end{proof}

We note some recurring themes in these additional adjacency solutions, which one could construe as another layer of unification between Cases. The ``chord'' edges and subsequent handle for connecting five vortices in Lemma~\ref{lem-k5} reappear in Lemma~\ref{lem-k6}. Proposition~\ref{prop-add-k14} is invoked in both Lemma~\ref{lem-k5} and \ref{lem-k8}. As mentioned earlier, most of the construction in Lemma~\ref{lem-k8} was applied to Case2-CG by Ringel and Youngs~\cite{RingelYoungs-Case2}.

It seems that nowhere in the literature, including in the original proof of the Map Color Theorem, is there a construction of a genus embedding of $K_n$ derived from an $(n,4)$-triangulation. Even though we outlined a natural approach in Proposition~\ref{prop-add-k14} for converting an $(n,4)$-triangulation to a genus embedding of $K_n$, no prior such unification was known. 

\subsection{Recasting handle operations}\label{sec-alternate}

Additional adjacency solutions are traditionally presented as a sequence of handles, which has the benefit of constructing some of the requisite minimum triangulations. However, when several handles are involved, it is not immediately apparent how such a construction was derived---Ringel~\cite{Ringel-MapColor} described the solution for Case 2-CG, which is largely identical to the one we used in Lemma~\ref{lem-k8}, as ``adventurous'' and ``much easier to understand than to discover.'' We can instead interpret these operations as generalized diamond sum-like operations~\cite{Bouchet-Diamond} using known embeddings. In our case, we make use of the embedding of $K_6$ in the torus formed by deleting a vertex from the triangular embedding of $K_7$, and the genus embedding of $K_8$ in the two-torus, where the two quadrilateral faces are incident with disjoint sets of vertices. Such an embedding appears in Ringel \cite[p.79]{Ringel-MapColor} and is reproduced in Appendix~\ref{app-sporadic}.

Recall that in Lemma~\ref{lem-k8}, the second, third, and fourth handles add all the remaining missing edges between lettered vertices, where all the activity takes place inside of the two quadrilateral faces formed from the first handle. Let $\phi: G \to S_g$ be the embedding of the graph after the first handle in Lemma~\ref{lem-k8}. Combining the next three handles into one step is equivalent to the following procedure, which is sketched in Figure~\ref{fig-diamond8}:

\begin{itemize}
\item Excise the interiors of the quadrilateral faces of $\phi$ and the aforementioned embedding $K_8 \to S_2$.
\item Identify the two embedded surfaces at their boundaries so that the two disjoint sets of four vertices become identified and the resulting surface is orientable.
\end{itemize}

\begin{figure}[ht]
\centering
\includegraphics[scale=1.2]{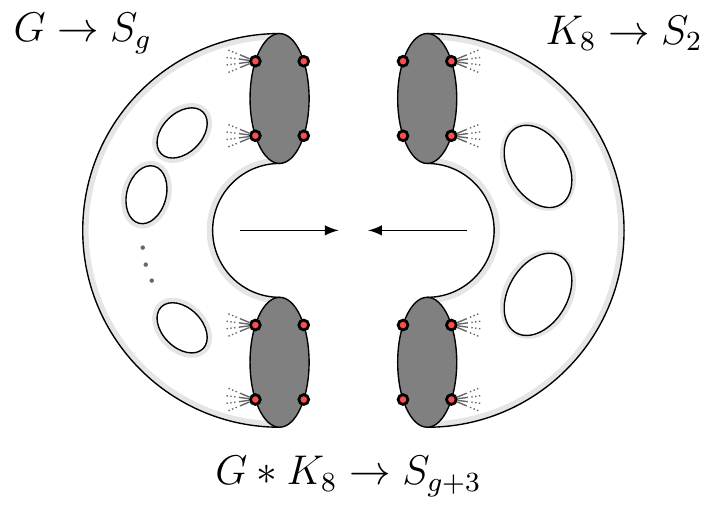}
\caption{Adding adjacencies between eight vertices with an embedding of $K_8$. Note that the genus increases by 3 since two boundary components are identified.}
\label{fig-diamond8}
\end{figure}

Hence the three handles are equivalent to a way of finding a genus embedding of $K_8$. We may also apply the same idea to reintepret the constructions in Lemma~\ref{lem-k5} and \ref{lem-k6} using the embedding of $K_6$. If, for example, we remove the edges $(\vorb, \vorc)$, $(\vorb,\vord)$, and $(\vorc, \vore)$ from Figure~\ref{fig-k6-add}, we have the hexagonal face $[\vora, \vord, \vorc, \vorf, \vore, \vorb]$. The goal of the last handle of the additional adjacency step in Lemma~\ref{lem-k6} is to add all the remaining edges between the lettered vertices, which we may accomplish by attaching the embedding of $K_6$ along this hexagonal face, as shown in Figure~\ref{fig-diamond7}.

\begin{figure}[ht]
\centering
\includegraphics[scale=1.2]{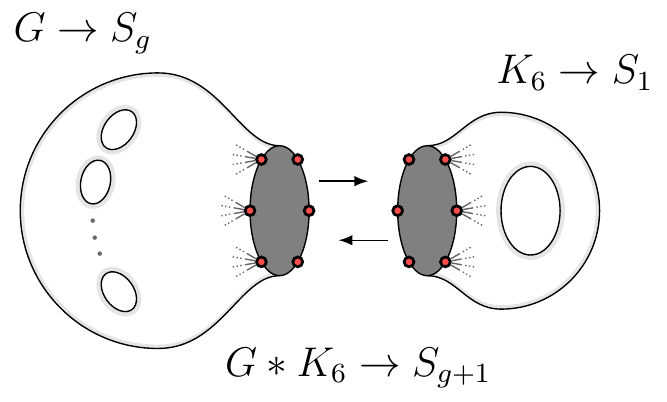}
\caption{An alternative way of adding the edges between six vertices using one handle.}
\label{fig-diamond7}
\end{figure}

\section{Index 3 current graphs}

We assume familiarity with current graphs, especially \S9 of Ringel~\cite{Ringel-MapColor}.  An \emph{index $k$ current graph} is a triple $(D, \phi, \alpha)$, where $D$ is a directed graph, $\phi: D \to S$ is a cellular $k$-face embedding of $D$ in an orientable surface $S$ and $\alpha: E(D) \to \Gamma$ is a labeling of each arc of $D$ with a \emph{current}, an element of a group $\Gamma$. In this paper, we only consider index 3 current graphs with cyclic current groups $\Gamma = \Z_{3m}$ for some integer $m$. Its three face boundary walks, which we call \emph{circuits}, are labeled $[0]$, $[1]$, and $[2]$.

The \emph{excess} of a vertex is the sum of the incoming currents minus the sum of the outgoing currents, and we say a vertex satisfies \emph{Kirchhoff's current law} (KCL) if its excess is 0. Vertices of degree 3 which do not satisfy KCL are called \emph{vortices}, which are each labeled with a bold lowercase letter. The \emph{log} of a circuit records the currents encountered along the walk in the following manner: if we traverse arc $e$ along its orientation, we write down $\alpha(e)$; otherwise, we write down $-\alpha(e)$; if we encounter a vortex, we record its label.

All of our index 3 current graphs with current groups $\Z_{3m}$ satisfy the following additional ``construction principles'', which are effectively the same as those in \S9.1 of Ringel~\cite{Ringel-MapColor}:

\begin{enumerate}
\item[(E1)] Each vertex is of degree 3 or 1.
\item[(E2)] The embedding has three circuits labeled $[0], [1], [2]$.
\item[(E3)] Each nonzero element $\gamma \in \Z_{3m}$ appears exactly once in the log of each circuit.
\item[(E4)] KCL is satisfied at every vertex of degree 3, except vortices, which are labeled with letters.
\item[(E5)] Every vortex is incident with all three circuits and has an excess which generates the subgroup of $\Z_{3m}$ consisting of the multiples of 3. 
\item[(E6)] If circuit $[a]$ traverses arc $e$ along its orientation and circuit $[b]$ traverses $e$ in the opposite direction, then $\alpha(e) \equiv b-a \pmod{k}$. 
\item[(E7)] The current on every arc incident with a vertex of degree 1 is of order 2 or 3 in $\Z_{3m}$.
\end{enumerate}

If all the construction principles are satisfied, the current graph generates a triangular embedding of the graph $K_{3m}+\overline{K_\ell}$, where $G+H$ is the graph join operation and $\ell$ is the number of vortices. Each element of $\Z_{3m}$ corresponds to a vertex in the complete graph $K_{3m}$, and each of the vortices provides an additional vertex, which is adjacent to all elements of $\Z_{3m}$, but none of the other vortex vertices. It is more common to think of the resulting graph instead as $K_{3m+\ell}-K_\ell$, which emphasizes the total number of vertices and the number of missing edges needed to form a complete graph. An example of an index 3 current graph is given in Figure~\ref{fig-current-c5s1}. The logs of its circuits are:

\begin{figure}[ht]
\centering
\includegraphics[scale=0.8]{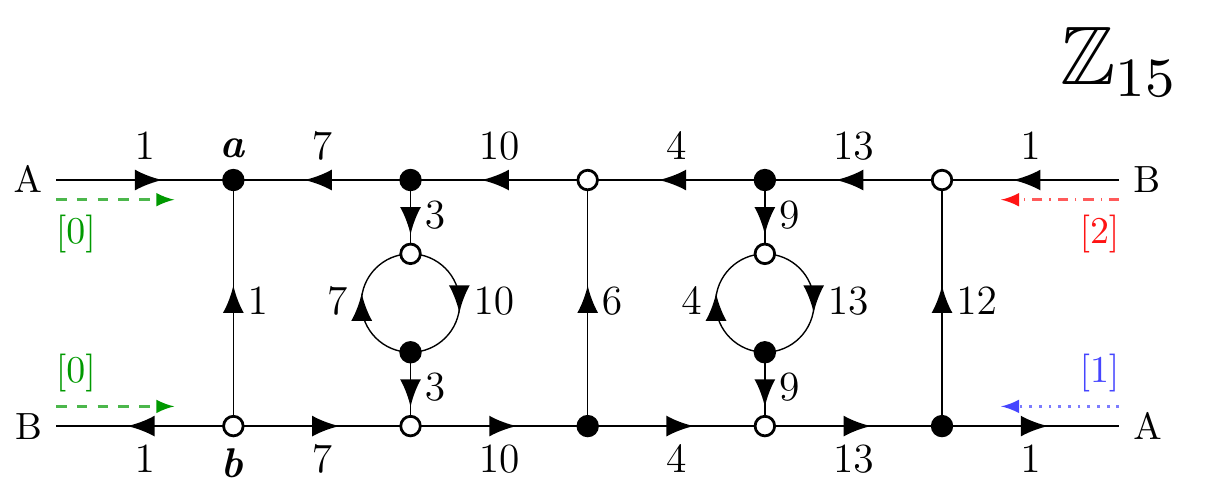}
\caption{A current graph for $K_{17}-K_2$. Solid and hollow vertices correspond to clockwise and counterclockwise rotations, respectively.}
\label{fig-current-c5s1}
\end{figure}

$$\begin{array}{rrrrrrrrrrrrrrrrrrrrrrrrrrrrrrrr}
\lbrack0\rbrack. & 1 & \vora & 8 & 5 & 9 & 4 & 13 & 12 & 14 & \vorb & 7 & 10 & 6 & 11 & 2 & 3 \\
\lbrack1\rbrack. & 14 & 2 & 6 & 4 & 13 & 9 & 11 & 5 & 12 & 7 & 10 & 3 & 8 & \vorb & 1 & \vora \\
\lbrack2\rbrack. & 1 & 13 & 9 & 11 & 2 & 6 & 4 & 10 & 3 & 8 & 5 & 12 & 7 & \vora & 14 & \vorb
\end{array}$$

To generate the embedding from the logs of these circuits, for each element $\gamma \in \Z_{3m}$ in the group, the rotation at vertex $\gamma$ is found by taking the log of circuit $[\gamma \bmod{k}]$ and adding $\gamma$ to each of its non-letter elements. The rotations at the numbered vertices thus read:

$$\begin{array}{rrrrrrrrrrrrrrrrrrrrrrrrrrrrrrrr}
0. & 1 & \vora & 8 & 5 & 9 & 4 & 13 & 12 & 14 & \vorb & 7 & 10 & 6 & 11 & 2 & 3 \\
1. & 0 & 3 & 7 & 5 & 14 & 10 & 12 & 6 & 13 & 8 & 11 & 4 & 9 & \vorb & 2 & \vora \\
2. & 3 & 0 & 11 & 13 & 4 & 8 & 6 & 12 & 5 & 10 & 7 & 14 & 9 & \vora & 1 & \vorb \\
3. & 4 & \vora & 11 & 8 & 12 & 7 & 1 & 0 & 2 & \vorb & 10 & 13 & 9 & 14 & 5 & 6 \\
4. & 3 & 6 & 10 & 8 & 2 & 13 & 0 & 9 & 1 & 11 & 14 & 7 & 12 & \vorb & 5 & \vora \\
5. & 6 & 3 & 14 & 1 & 7 & 11 & 9 & 0 & 8 & 13 & 10 & 2 & 12 & \vora & 4 & \vorb \\
6. & 7 & \vora & 14 & 11 & 0 & 10 & 4 & 3 & 5 & \vorb & 13 & 1 & 12 & 2 & 8 & 9 \\
7. & 6 & 9 & 13 & 11 & 5 & 1 & 3 & 12 & 4 & 14 & 2 & 10 & 0 & \vorb & 8 & \vora \\
8. & 9 & 6 & 2 & 4 & 10 & 14 & 12 & 3 & 11 & 1 & 13 & 5 & 0 & \vora & 7 & \vorb \\
9. & 10 & \vora & 2 & 14 & 3 & 13 & 7 & 6 & 8 & \vorb & 1 & 4 & 0 & 5 & 11 & 12 \\
10. & 9 & 12 & 1 & 14 & 8 & 4 & 6 & 0 & 7 & 2 & 5 & 13 & 3 & \vorb & 11 & \vora \\
11. & 12 & 9 & 5 & 7 & 13 & 2 & 0 & 6 & 14 & 4 & 1 & 8 & 3 & \vora & 10 & \vorb \\
12. & 13 & \vora & 5 & 2 & 6 & 1 & 10 & 9 & 11 & \vorb & 4 & 7 & 3 & 8 & 14 & 0 \\
13. & 12 & 0 & 4 & 2 & 11 & 7 & 9 & 3 & 10 & 5 & 8 & 1 & 6 & \vorb & 14 & \vora \\
14. & 0 & 12 & 8 & 10 & 1 & 5 & 3 & 9 & 2 & 7 & 4 & 11 & 6 & \vora & 13 & \vorb
\end{array}$$

The rotation around each lettered vertex is ``manufactured'' so that the entire embedding is triangular and orientable. To facilitate this process, we make use of the following characterization of triangular embeddings:

\begin{proposition}[{e.g., Ringel~\cite[\S 2.3]{Ringel-MapColor}}]
An embedding of a simple graph $G$ is triangular if and only if for all vertices $i,j,k$, if the rotation at vertex $i$ is of the form
$$\begin{array}{rrrrrrrrrrrrrrrrrrrrrrrrrrrrrrrr}
i. & \dots & j & k & \dots,
\end{array}$$
then the rotation at vertex $j$ is of the form
$$\begin{array}{rrrrrrrrrrrrrrrrrrrrrrrrrrrrrrrr}
j. & \dots & k & i & \dots
\end{array}$$
\label{prop-ruler}
\end{proposition}

From the partial rotation system we have built up so far, we can determine the rotations at the remaining vortex vertices:

$$\begin{array}{rrrrrrrrrrrrrrrrrrrrrrrrrrrrrrrr}
\vora. & 0 & 1 & 2 & 9 & 10 & 11 & 3 & 4 & 5 & 12 & 13 & 14 & 6 & 7 & 8 \\
\vorb. & 0 & 14 & 13 & 6 & 5 & 4 & 12 & 11 & 10 & 3 & 2 & 1 & 9 & 8 & 7
\end{array}$$

The final embedding is a triangular one of $K_{17}-K_2$, which is a $(17,1)$-triangulation. It can be augmented into a genus embedding of $K_{17}$ using Proposition~\ref{prop-add-k2}.

The group we use for most of our constructions, including all infinite families, is $\Z_{12s+3}$. By combining construction principles (E6) and (E7), we find that in order to have a degree 1 vertex using this group, it must be the case that $s \equiv 2 \pmod{3}$. Thus, we only make use of degree 1 vertices and principle (E7) in a few \emph{ad hoc} constructions. 

The increased flexibility acquired from using index 3 current graphs is crucial. Since vortices have the same degree as other vertices, one can tweak the number of vortices while keeping the number of total vertices and edges fixed, i.e., one cannot rule out the existence of such current graphs using divisibility conditions on the numbers of vertices and edges alone. Furthermore, the conditions in Lemma~\ref{lem-k5}, i.e., having all five vortices lined up nearly consecutively, is only possible for current graphs with index at least 3. For index 1 and 2, such a configuration would violate a ``global'' KCL condition.  

A sketch of the standard proof of Case 5-CG (see Ringel~\cite[\S9.2]{Ringel-MapColor} or Youngs~\cite{Youngs-3569}) is given first, as we reuse significant parts of its structure for our current graphs. The case $s = 1$ was given earlier in Figure~\ref{fig-current-c5s1}, and the higher order cases are given in Figures~\ref{fig-current-c5s2} and \ref{fig-current-c5sgen}. The construction also works trivially for $s=0$ as well. 

\begin{figure}[ht]
\centering
\includegraphics[scale=0.8]{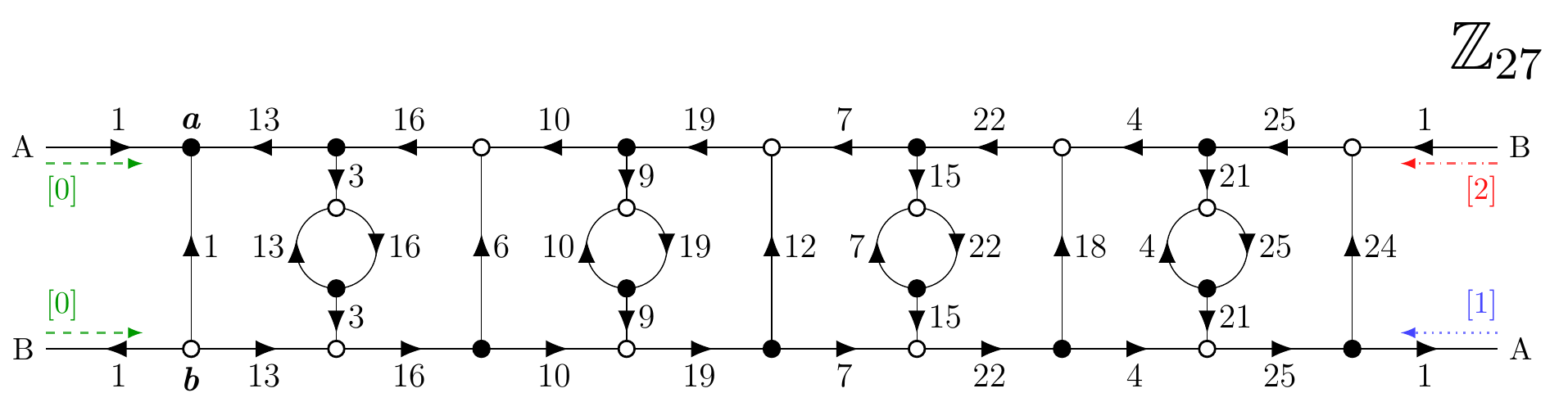}
\caption{A current graph for $K_{29}-K_2$.}
\label{fig-current-c5s2}
\end{figure}

\begin{figure}[ht]
\centering
\includegraphics[scale=0.8]{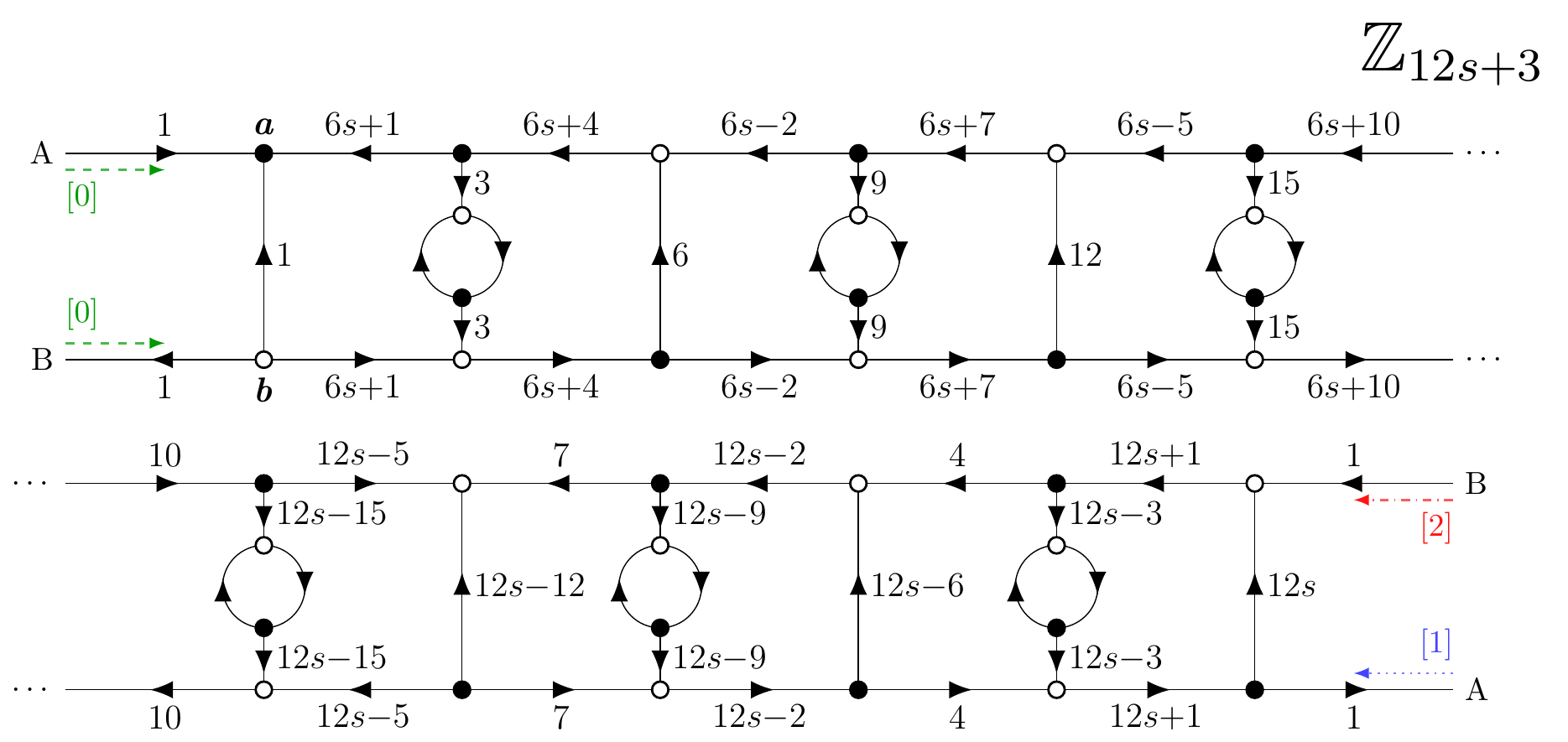}
\caption{The family of current graphs for $K_{12s+5}-K_2$, for general $s$. The omitted current on a circular arc is the same as those on the horizontal arcs above and below it.}
\label{fig-current-c5sgen}
\end{figure}

The general shape of the family of current graphs is a long ladder whose ``rungs'' alternate between simple vertical arcs and so-called ``globular rungs,'' where two vertices have a pair of parallel edges between them. As we parse from left to right, the vertical arcs, except for the arc connecting the two vortices, alternate in direction and form the arithmetic sequence consisting of the nonzero multiples of 3 in $\Z_{12s+3}$. The zigzag pattern induced on the horizontal arcs is essentially the canonical graceful labeling of a path graph on $4s{+}1$ vertices (see, e.g., Goddyn \emph{et al.} \cite{Goddyn-Exponential} for more information on this connection), where the vertical arcs correspond to the edge labels on the path graph. The horizontal arcs come in pairs that share the same current and are oriented in opposite directions. The currents on these arcs exhaust all the elements of the form $3k{+}1$ in $\Z_{12s+3}$. 

To see that construction principle (E3) is satisfied, the circuit $[0]$ traverses each pair of horizontal arcs twice in the rightward direction, so each element $3k{+}1$ and its inverse appear in the log of the circuit. Circuits $[1]$ and $[2]$ pass through only one arc of each pair of horizontal arcs---they each pass through the inverse of that current on one of the parallel arcs in a nearby globular rung. For the multiples of 3, note that for each such element and its inverse, exactly one of them appears as vertical arcs on a globular rung, and one appears on a simple rung. For the former, both circuits $[1]$ and $[2]$ will make use of such arcs in both directions, and for the latter, circuit $[0]$ will pass through in both directions. 

We utilize this family of current graphs in the following way: for the general cases of current graph constructions, they all consist of 
\begin{itemize}
\item A \emph{fixed portion}, which contains vortices and some salient currents for additional adjacency solutions. The underlying directed graph stays the same, while the currents may vary as a function of $s$. 
\item A \emph{varying portion}, which subsumes all remaining currents not present in the fixed portion. The size of this ingredient varies as a function of $s$, and the currents are arranged in a straightforward pattern.
\end{itemize}

In the construction for Case 5, we might consider the vortices and its incident edge ends as the fixed portion, and the rest of the graph (see Figure~\ref{fig-boseladder}) as the varying portion. The solutions for Case 3 and 5 of the Map Color Theorem is, conincidentally or not, intimately connected to Bose's construction for Steiner triple systems on $6s{+}3$ elements (see Grannell~\emph{et al.}~\cite{Grannell-SurfaceEmbeddings}) and are prized for their simplicity. For these reasons, we consider this varying portion, which we call the \emph{Bose ladder}, to be the best possible choice for index 3 current graphs. 

The approach of Youngs~\emph{et al.}~\cite{Youngs-3569, GuyYoungs-Smooth, GuyRingel} was to first finalize the fixed portion and then solve certain labeling problems (so-called ``zigzag'' and ``chord'' problems) to deal with the varying portion. We tackle the problem in reverse, opting to massage the fixed portion around a preset varying portion, which we choose to be a contiguous subset of the Bose ladder. Starting with the arc labeled $1$ that runs between the two vortices, we successively peel off rungs of the Bose ladder until we have enough material for our desired fixed portion. 

\begin{figure}[ht!]
\centering
\includegraphics[scale=0.8]{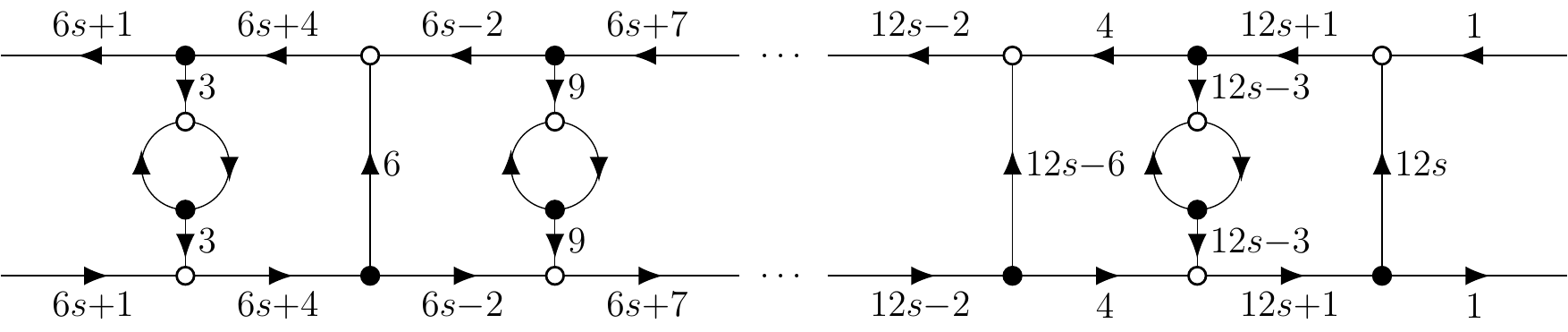}
\caption{The Bose ladder is essentially the current graphs for Case 5 with two vertices deleted.}
\label{fig-boseladder}
\end{figure}

We expect this procedure to become more difficult as the number of vortices increases---not only do we need appropriate currents that feed into the vortices, but there becomes an imbalance between the currents which are not divisible by 3 and those which are. Each vortex will use three currents of the former type, leaving a surplus of those of the latter type. To correct this effect, we make use of the \emph{double bubble} in Figure~\ref{fig-doublebubble}, which is essentially two globular rungs joined together. By tracing out the partial circuits and invoking construction principle (E6), we find that all six currents entering the highest and lowest vertices must be divisible by 3, while the four remaining arcs may be labeled with an element not divisible by 3 depending on which circuits touch this gadget. The double bubble and its generalization have appeared in other work regarding current graphs of index greater than 1, such as Korzhik and Voss~\cite{KorzhikVoss} and Pengelley and Jungerman~\cite{Pengelley-Index4}.

In all of our current graph constructions, we use the cyclic group $\mathbb{Z}_{12s+3}$ unless we specify otherwise. While we often simplify the labels by reversing the directions of some arcs, e.g. replacing a label like $12s{+}1$ with $2$, the ends which connect to the Bose ladder are kept unchanged, i.e., as a current which is congruent to $1 \pmod{3}$. 

\begin{figure}[ht!]
\centering
\includegraphics[scale=0.85]{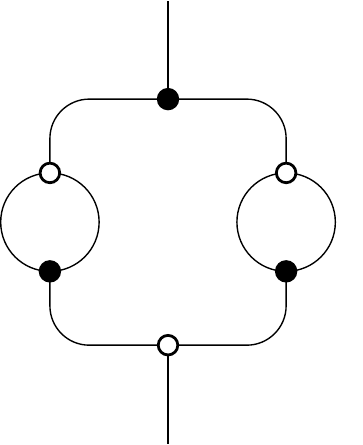}
\caption{The ``double bubble'' motif appears in all of our general constructions.}
\label{fig-doublebubble}
\end{figure}

\section{Handle subtraction for minimum triangulations}

The forthcoming embeddings $K_{12s+3+k}-K_k$ and the embeddings \emph{en route} to constructing a genus embedding of $K_{12s+3+k}$ already constitute minimum triangulations, namely 
$$\left(12s{+}3{+}k, \binom{k}{2}{-}6h\right)\textrm{-triangulations,}$$
where $h$ is a nonnegative integer less than the number of added handles. To construct minimum triangulations on the same number of vertices, but with more missing edges, we turn to the main idea of Jungerman and Ringel~\cite{JungermanRingel-Minimal}: we enforce a specific structure in the current graph that allows us to ``subtract'' handles. The fragment shown in Figure~\ref{fig-subtract} is what we refer as an \emph{arithmetic 3-ladder}. If the step size $h$ in the arithmetic sequence is divisible by $3$ (more generally, divisible by the index of the current graph), then it is possible to find triangular embeddings in smaller-genus surfaces in the following manner:

\begin{figure}[ht]
\centering
\includegraphics[scale=0.9]{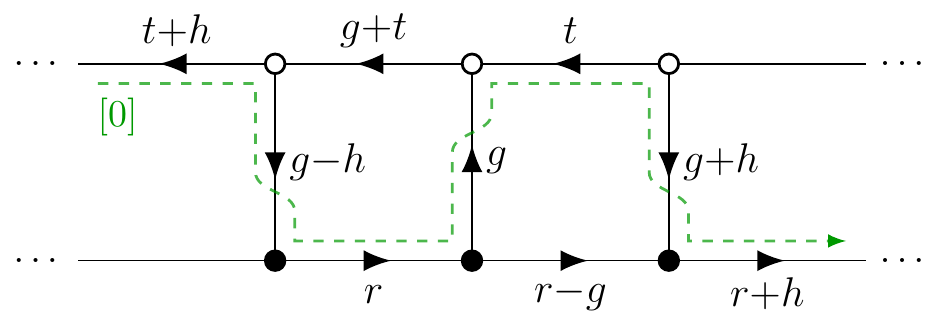}
\caption{An arithmetic 3-ladder and a circuit passing through it.}
\label{fig-subtract}
\end{figure}

\begin{lemma}[Jungerman and Ringel~\cite{JungermanRingel-Minimal}]
Let $(D, \phi, \alpha)$ be an index 3 current graph with current group $\Z_{3m}$ that satisfies all construction principles. Suppose further that it has an arithmetic 3-ladder with step size divisible by 3. If the derived embedding of the current graph has $|V|$ vertices and $|E|$ edges, then for each $k = 0, \dotsc, m$, there exists a triangular embedding of a graph with $|V|$ vertices and $|E|-6k$ edges. 
\end{lemma}
\begin{proof}[Proof Sketch]
Following Figure~\ref{fig-subtract}, the rotation at vertices 0 and $h$ are of the form
$$\begin{array}{ccccccccccccccccccccccccccccccccccccccccccc}
0. & \dots & -t{-}h & g{-}h & r & g & -t & g{+}h & r{+}h & \dots \\
h. & \dots & -t & g & r{+}h & g{+}h & & & & \dots 
\end{array}$$
Here we used the fact that $h$ is divisible by 3. We may infer, by repeated application of Proposition~\ref{prop-ruler}, the following partial rotation system, for $i = 0, 1, \dotsc, m$:
\begin{equation}
\begin{array}{ccccccccccccccccccccccccccccccccccccccccccc}
0. & \dots & g & -t & g{+}h & r{+}h & \dots \\
g. & \dots & r{+}h & h & -t & 0 & \dots \\ 
r{+}h. & \dots & 0 & g{+}h & h & g & \dots \\ 
\\
h. & \dots & -t & g & r{+}h & g{+}h & \dots \\
-t. & \dots & g{+}h & 0 & g & h & \dots \\
g{+}h. & \dots & h & r{+}h & 0 & -t & \dots 
\end{array}
\label{eq-handle}
\end{equation}
If we delete the middle two columns, the rotation system becomes
\begin{equation*}
\begin{array}{ccccccccccccccccccccccccccccccccccccccccccc}
0. & \dots & g &  r{+}h & \dots \\
g. & \dots & r{+}h & 0 & \dots \\ 
r{+}h. & \dots & 0 & g & \dots \\ 
\\
h. & \dots & -t & g{+}h & \dots \\
-t. & \dots & g{+}h & h & \dots \\
g{+}h. & \dots & h & -t & \dots 
\end{array}
\end{equation*}
This new embedding has six fewer edges, and is still triangular by Proposition~\ref{prop-ruler}, hence it must be a triangular embedding on a surface with one fewer handle by Proposition~\ref{prop-triangle}. 

More generally, we obtain other handles that can be subtracted in the same manner, using the additivity rule. That is, we can find another subtractible handle by adding a multiple of $3$ to every element of (\ref{eq-handle}). The six edges from each of $m$ handles can be deleted simultaneously, as none of the handles share any faces. 
\end{proof}

One way to visualize this operation is to interpret it as the reverse of Construction~\ref{cons-mergehandle}, like in Figure~\ref{fig-polysubtract}. One can check that in all instances in this paper, the number of handles we can subtract in a given embedding is greater than the number needed to realize the minimum triangulation with the fewest number of edges, i.e., the $(n,t)$-triangulation where $t \approx n-6$.

\begin{figure}[ht]
\centering
\includegraphics{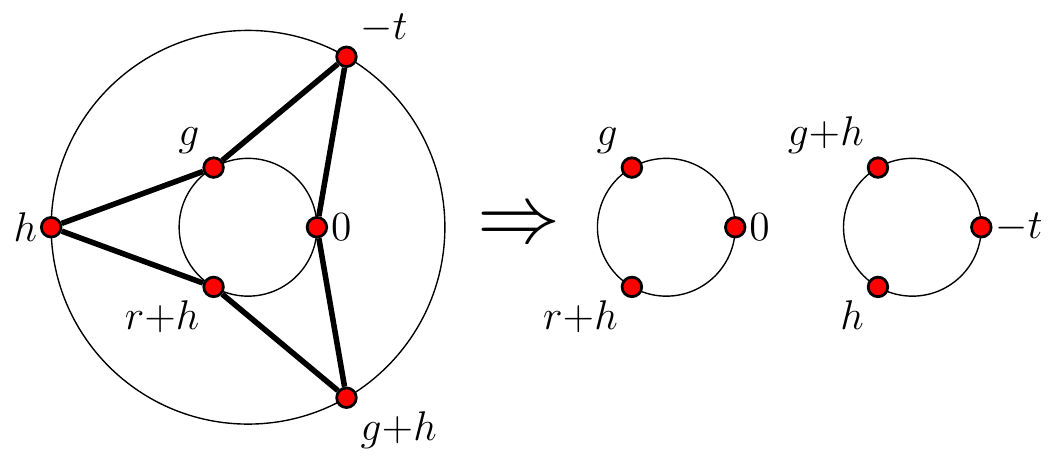}
\caption{The six deleted edges form a cycle that is, roughly speaking, surrounded by two triangles.}
\label{fig-polysubtract}
\end{figure}

\subsection{Comparison with existing literature}

Our utilization of index 3 current graphs is rooted in Jungerman and Ringel's \cite{JungermanRingel-Minimal} solution for Case 5-MT as a straightforward modification of the current graphs used for Case 5-CG. We make no improvement here, but use a variation of their construction as an example of the infinite families of current graphs we seek. 

The standard approach to Case 6-CG is to use index 3 current graphs to first obtain a triangular embedding of $K_{12s+6}-K_3$. The general solution Ringel~\cite[\S9.3]{Ringel-MapColor} chose to present works for all $s \geq 4$, and for $s = 2$, an \emph{ad hoc} current graph that makes use of construction principle (E7) is shown. Jungerman and Ringel~\cite{JungermanRingel-Minimal} solved the remaining minimum triangulations using two families of index 1 current graphs. For $s=1$, the case of $(18,3)$-triangulations is particularly difficult---Jungerman~\cite{Jungerman-K18} found a triangular embedding of $K_{18}-K_3$ using computer search, and we believe that such an embedding cannot be constructed with index 3 or lower current graphs (see the discussion in Section~\ref{subsec-case6} and Appendix~\ref{app-case6}). In \cite{Sun-Index2}, the author unifies the $(18,9)$-triangulation and genus embedding of $K_{18}$ cases using a somewhat \emph{ad hoc} index 2 current graph. 

Index 1 embeddings of $K_{12s+8}-K_5$ were apparently known to Ringel and Youngs (see Ringel~\cite[p.86]{Ringel-MapColor}), though they were unable to use them find genus embeddings of $K_{12s+8}$. Instead, Jungerman and Ringel~\cite{JungermanRingel-Minimal} used them for most of the minimum triangulations on $12s{+}8$ vertices, i.e., the $(12s{+}8,10+6h)$-triangulations for nonnegative $h$. For the remaining $(12s{+}8,4)$-triangulation case, they found two families of index 2 current graphs that could be modified into an embedding of $K_{12s+8}-(K_2\cup P_3)$. 

The best solution for Case 9-CG is a beautiful construction of Jungerman, but it does not construct minimum triangulations except for the exceptional surface $S_2$. For the general case, a family of current graphs found by Guy and Ringel~\cite{GuyRingel} produced\footnote{There are two errors in Figure 1 of~\cite{GuyRingel}: the top left current should be ``$6s{+}1$'' and the vertex between ``$x$'' and ``$z$'' should be a vortex labeled ``$w$''.} minimum triangulations for all $s \geq 5$. Jungerman and Ringel~\cite{JungermanRingel-Minimal} supplied the remaining cases via a variety of approaches, primarily using an inductive construction where some triangular embeddings are glued to one another.

The only previously known solution for the genus of $K_{12s+11}$ for $s \geq 1$ is that of Ringel and Youngs~\cite{RingelYoungs-Case11} for $s \geq 2$ and the \emph{ad hoc} embedding of Mayer~\cite{Mayer-Orientables} for $s = 1$. In the general case, Ringel and Youngs start with an embedding of $K_{12s+11}-K_5$, where the missing edges are added using a highly tailored additional adjacency step. The same current graph yields minimum triangulations of type $(12s{+}11, 10{+}6h)$ for $h \geq 0$, but the troublesome case of $(12s{+}11, 4)$-triangulations, like in Case 8-MT, was resolved via two complicated families of index 2 current graphs. 

Our approach gives a unified construction for both the Map Color Theorem and the minimum triangulations problem for Cases 6, 8, 9, and 11. The infinite families of current graphs covers all $s \geq 2$ for Cases 6, 8, and 9, and $s \geq 3$ for Case 11. In all these solutions, we use families of index 3 current graphs whose varying portions are a part of the Bose ladder. One attractive property of using index 3 current graphs is that we are able to give a solution that does not break into two parts depending on the parity of $s$, as was the case in Jungerman and Ringel's \cite{JungermanRingel-Minimal} current graphs for Case 6-, 8-, and 11-MT. For Cases 9 and 11, we give \emph{ad hoc} constructions for smaller values of $s$. Of particular interest is the case of $n = 23$, for which we give the first current graph construction for a genus embedding of $K_{23}$. 

We present the constructions in increasing difficulty of the additional adjacency solution. In particular, Case 9, which has six vortices, is ultimately simpler than Case 8 because of the additional constraint needed in Lemma~\ref{lem-k5}.

\subsection{Case 5}

As a warmup, let us consider how to find minimum triangulations for Case 5. The original solution in Figure~\ref{fig-current-c5sgen} does not have any arithmetic 3-ladders, but we can modify it by swapping two of the rungs in the Bose ladder, namely the two with vertical arcs labeled $6$ and $12s{-}3$, as in Figure~\ref{fig-case5-minimal}. In this drawing and all forthcoming figures, we only describe the fixed portion of the family of current graphs---at the ellipses, we complete the picture by attaching the corresponding segment of the Bose ladder, as mentioned earlier. Exactly where to truncate the Bose ladder is determined by the currents at the ends of the fixed portion. 

\begin{figure}[ht!]
\centering
\includegraphics[scale=0.85]{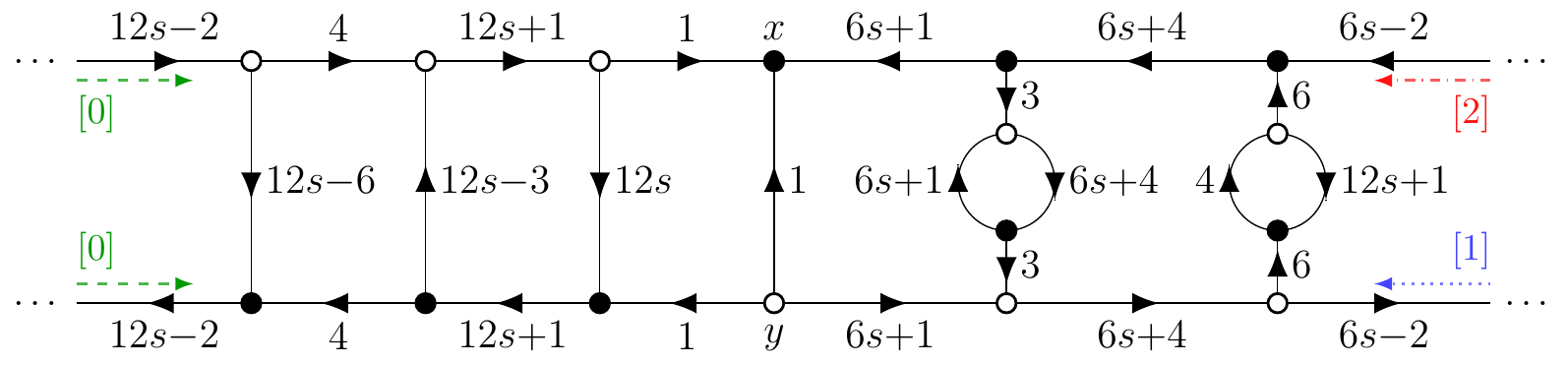}
\caption{A slight modification to the Bose ladder that produces minimum triangulations.}
\label{fig-case5-minimal}
\end{figure}

The idea of pairing the rungs is crucial in Youngs' method~\cite{Youngs-3569} for constructing index 3 current graphs. In their proof of minimum triangulations for Case 5, Jungerman and Ringel~\cite{JungermanRingel-Minimal} took this idea to the extreme and switched all pairs of rungs so that all of the globular rungs appeared on one side of the ladder, but as seen in our example, implementing all these exchanges is not necessary. 

We note that to the left of the vortices in our drawing in Figure~\ref{fig-case5-minimal}, the directions of the arcs are inverted from that of Figure~\ref{fig-current-c5sgen}. Most of our infinite families (except the alternate Case 6-CG construction in Appendix \ref{app-case6}) involve attaching a Bose ladder with a ``M\"obius twist," i.e., the final current graph is a long ladder-like graph whose top-left and bottom-left ends become identified with the bottom-right and top-right ends, respectively. 

\subsection{Case 6}\label{subsec-case6} 

The family of current graphs in Figure~\ref{fig-case6} applies for all $s \geq 2$ and has an arithmetic 3-ladder, giving a simpler and more unified construction for Case 6-CG (after applying Proposition~\ref{cor-k3}), in addition to providing a single family of current graphs, irrespective of parity, for Case 6-MT. The case $s = 1$ is particularly pesky---in the original proof of the Map Color Theorem, the minimum genus embedding of $K_{18}$ was found using purely \emph{ad hoc} methods by Mayer~\cite{Mayer-Orientables}. One might ask if an index 3 current graph exists for $K_{18}-K_3$, but an exhaustive computer search suggests that one does not exist. In Appendix~\ref{app-case6}, we present another solution for Case 6-CG, $s \geq 2$, that almost achieves the $18$-vertex case.

\begin{figure}[ht!]
\centering
\includegraphics[scale=0.85]{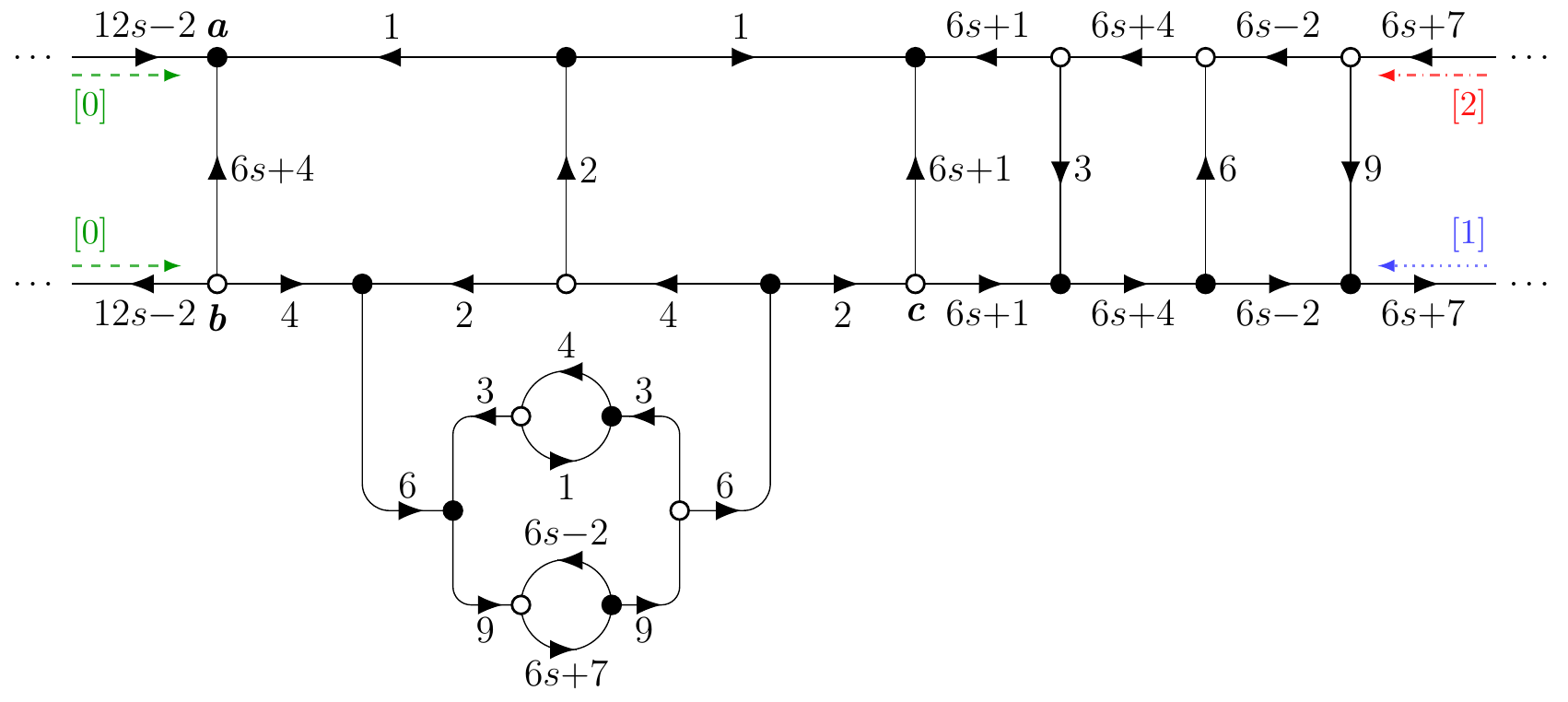}
\caption{A current graph for $K_{12s+6}-K_3$ for $s \geq 2$.}
\label{fig-case6}
\end{figure}

\subsection{Case 9}
We improve on the construction of Guy and Ringel~\cite{GuyRingel} with the family of index 3 current graphs seen in Figure~\ref{fig-case9}. These current graphs produce triangular embeddings of $K_{12s+9}-K_6$ for all $s \geq 2$, and the vertical rungs labeled $3, 6, 9$ form an arithmetic 3-ladder. The circuits $[1]$ and $[2]$ have the six vortices packed as close together as possible. In particular, the log of circuit $[1]$ reads
$$
\begin{array}{ccccccccccccccccccccccccccccccccccccccccccc}
[1]. & \dots & \vora & 4 & \vorb & \dots & \vorc & 1 & \vord & \dots & \vore & 12s{+}1 & \vorf & \dots,
\end{array} 
$$ 
so we may apply Lemma~\ref{lem-k6} with, e.g., $u = 1$, to obtain $(12s{+}9,9)$- and $(12s{+}9,3)$-triangulations and a genus embedding of $K_{12s+9}$. 

\begin{figure}[ht!]
\centering
\includegraphics[scale=0.85]{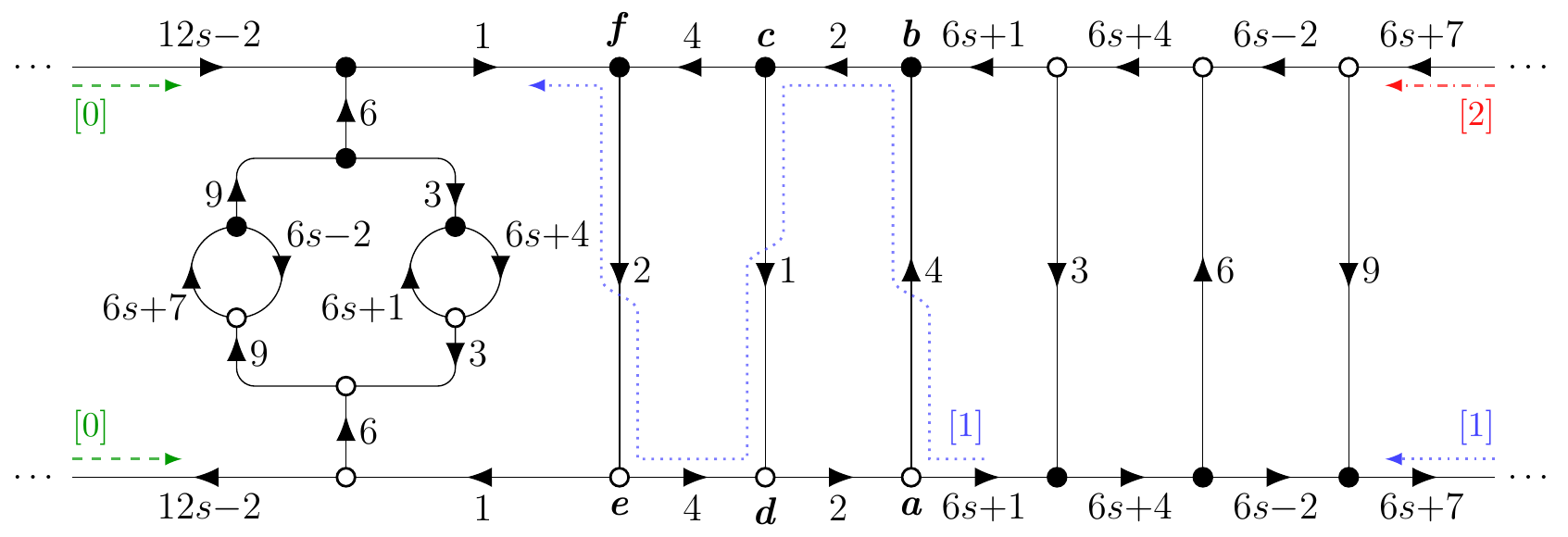}
\caption{A current graph for $K_{12s+9}-K_6$ for $s \geq 2$. Additional fragments of circuits besides the guidelines at the left and right ends indicate components used in the additional adjacency solution.}
\label{fig-case9}
\end{figure}

For the case $s = 1$, Appendix~\ref{app-case911} contains an index 3 current graph with an arithmetic 3-ladder that yields a triangular embedding of $K_{21}-K_3$. The remaining case $s = 0$ is the lone exception to Theorem~\ref{thm-mt}. Huneke \cite{Huneke-Minimum} proved that no triangulation of the surface $S_2$ has 9 vertices, so the embedding of $K_8$ in $S_2$ with its quadrilateral faces subdivided (see Appendix~\ref{app-sporadic}) is a minimum triangulation on 10 vertices. Adding an edge between these two subdivision vertices with Construction~\ref{cons-mergehandle} and immediately contracting that edge results in a genus embedding of $K_9$.

\subsection{Case 8}

The family of current graphs in Figure~\ref{fig-current-c8-sgen} yields triangular embeddings of $K_{12s+8}-K_5$ and has the necessary arithmetic 3-ladder for producing the minimum triangulations on fewer edges. The logs of this current graph are of the form
$$
\begin{array}{ccccccccccccccccccccccccccccccccccccccccccc}
\lbrack0\rbrack. & \dots & 6s{+}1 & 12s & \dots & 12s{-}3 & 6s{-}2 & \dots
\end{array} 
$$
$$
\begin{array}{ccccccccccccccccccccccccccccccccccccccccccc}
\lbrack2\rbrack. & \dots & \vora & 6s{+}2 & \vorb & 12s{+}1 & \vorc & 6s{-}1 & \vord & 12s{-}2 & \vore & \dots 
\end{array}
$$
These translate, by additivity, to the rotations 
$$
\begin{array}{ccccccccccccccccccccccccccccccccccccccccccc}
3. & \dots & 6s{+}4 & 0 & \dots & 12s & 6s{+}1 & \dots
\end{array} 
$$ 
$$
\begin{array}{ccccccccccccccccccccccccccccccccccccccccccc}
2. & \dots & \vora & 6s{+}4 & \vorb & 0 & \vorc & 6s{+}1 & \vord & 12s & \vore & \dots 
\end{array}
$$
By applying Lemma~\ref{lem-k5} with $u = 2$, $v=3$, $(p_1,p_2,p_3,p_4) = (6s{+}5, 0, 6s{+}1, 12s)$, we can construct a $(12s{+}8,4)$-triangulation and a genus embedding of $K_{12s+8}$.

\begin{figure}[ht!]
\centering
\includegraphics[scale=0.85]{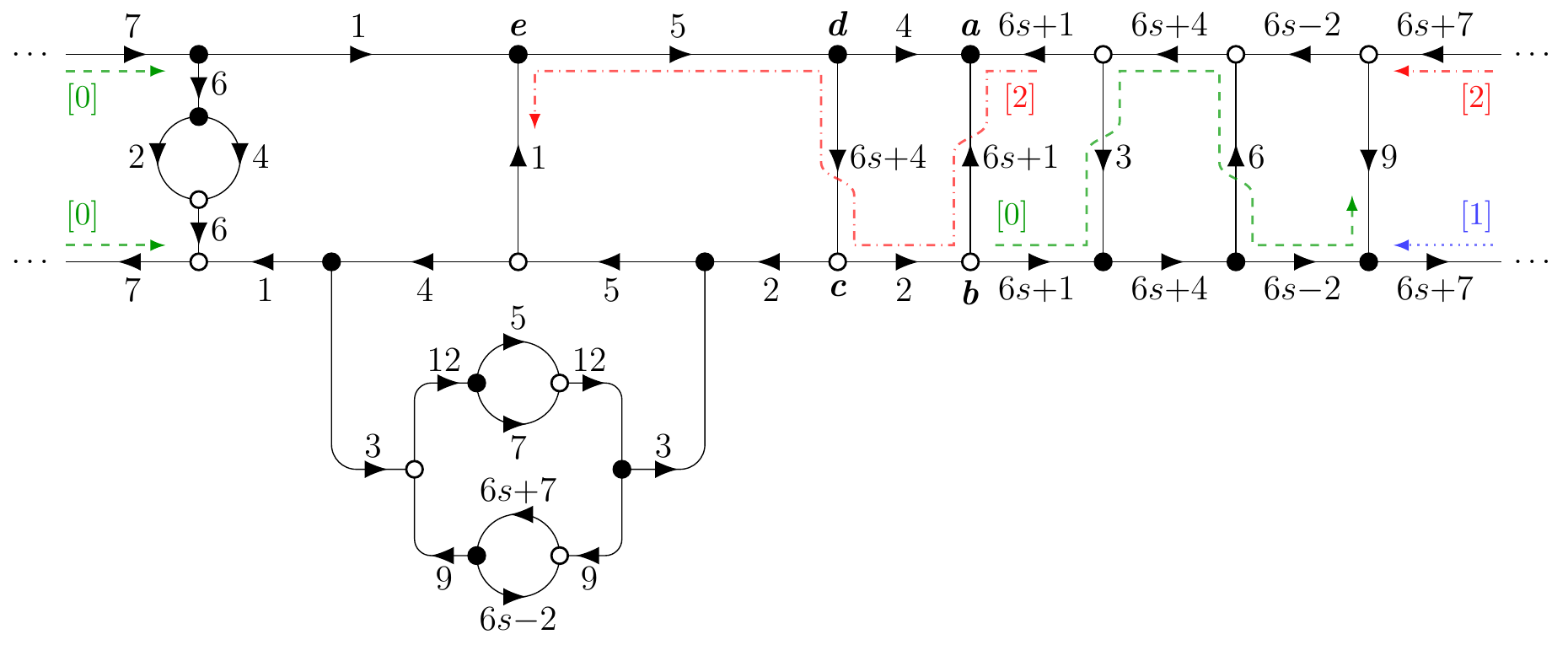}
\caption{A family of index 3 current graphs for $K_{12s+8}-K_5$, $s \geq 2$.}
\label{fig-current-c8-sgen}
\end{figure}

\begin{remark}
Our additional adjacency solution makes use of some of the arcs forming the arithmetic 3-ladder. However, there is no conflict since handle subtractions and additional adjacencies are never applied simultaneously. 
\end{remark}

\subsection{Case 11}
\label{sec-case11}

For $s \geq 3$, we found the family of current graphs in Figure~\ref{fig-current-c11-sgen} that generate triangular embeddings of $K_{12s+11}-K_8$. On the bottom right is an arithmetic 3-ladder with labels $9, 12, 15$. By examining the circuit $[1]$, we obtain the rotations
$$
\begin{array}{ccccccccccccccccccccccccccccccccccccccccccc}
1. & \dots & \vora & 6s{+}8 & \vorb & \dots & \vorc & 5 & \vord & \dots & \vore & 12s{-}1 & \vorf & \dots & \vorg & 6s{+}2 & \vorh & \dots
\end{array} 
$$ 
$$
\begin{array}{ccccccccccccccccccccccccccccccccccccccccccc}
12s{+}1. & \dots & 6s{+}8 & 6s{+}2 & \dots & 5 & 12s{-}1 & \dots 
\end{array}
$$
Applying Lemma~\ref{lem-k8} with $u = 1, v = 12s{+}1, (p_1,p_2,p_3,p_4) = (6s{+}8, 5, 12s{-}1, 6s{+}2)$ yields the remaining triangulations and a genus embedding of $K_{12s+11}$, $s \geq 3$. 
\begin{figure}[ht!]
\centering
\includegraphics[scale=0.9]{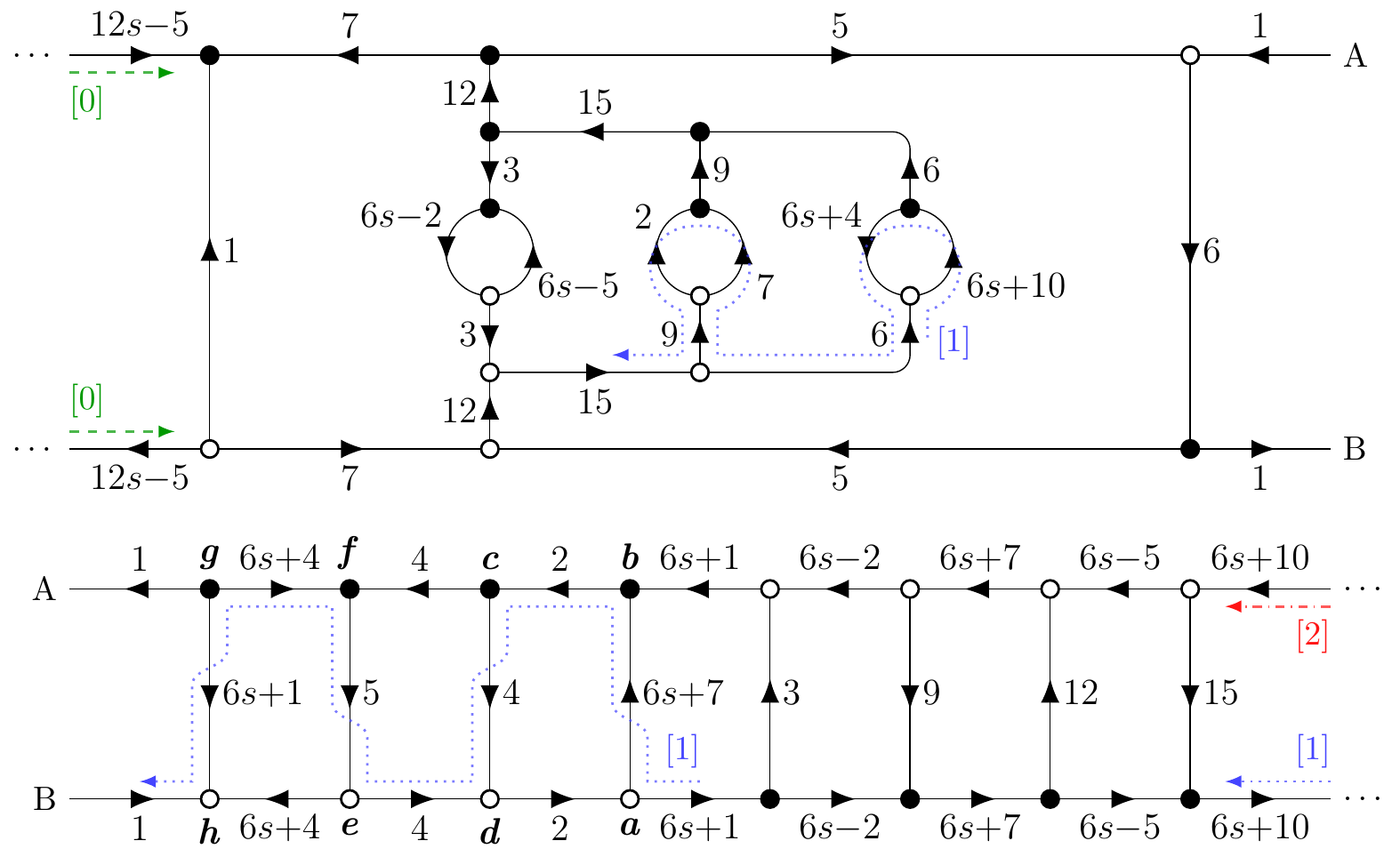}
\caption{Index 3 current graphs for $K_{12s+11}-K_8$, $s \geq 3$.}
\label{fig-current-c11-sgen}
\end{figure}

 The special cases $s = 1,2$ have current graphs found in Appendix~\ref{app-case911}, and an \emph{ad hoc} rotation system for $s = 0$ is given in Appendix~\ref{app-sporadic}.

\section{Conclusion}

We found index 3 constructions that produced simultaneous solutions to the genus of the complete graphs and
to minimum triangulations of surfaces, for Cases 6, 8, 9, and 11:

\begin{itemize}
\item Two constructions were presented for Case 6, $s \geq 2$ of the Map Color Theorem. Prior to the present paper, the only previously known current graph for $s = 2$ was not generalizable to higher values of $s$ due to its use of construction principle (E7). 
\item A significantly simpler solution was found for $K_{12s+9}-K_6$ than that of Guy and Ringel~\cite{GuyRingel} that also works for $s = 2,3$. 
\item Unified constructions for Cases 8 and 11. For the latter, they are the first known triangular embeddings of $K_{12s+11}-K_8$ for $s \geq 3$, and the case $s = 1$ for the Map Color Theorem now has a solution using current graphs. The additional adjacency solution for Case 11 (Lemma~\ref{lem-k8}) is more straightforward than the original construction by Ringel and Youngs~\cite{RingelYoungs-Case11}, especially in light of the interpretation using diamond sum-like operations. 
\end{itemize}

As mentioned earlier, index 3 current graphs allow for changing the number of vortices without violating divisibility conditions necessary for the existence of current graphs. We expect that for fixed $k > 1$ and sufficiently large $s$, there exist appropriate current graphs for triangular embeddings of $K_{12s+3+k}-K_{k}$. The results of this paper extend the applicability of index 3 current graphs to roughly half of both of the Map Color Theorem and the minimum triangulations problem, and we believe that complete solutions for a sufficiently large number of vertices are possible by extending the results presented here. 

We made use of the current group $\Z_{12s+3}$ in our infinite families of current graphs, reserving the group $\Z_{12s+6}$ for the special cases presented in the Appendix \ref{app-case911}. We were unable to find triangular embeddings of $K_{12s+11}-K_8$ for $s=1,2$, so we resorted to using a different approach for these cases. An open problem would be to find an analogue of the Bose ladder for the latter group---one tricky aspect is incorporating the order 2 element $6s{+}3$ into such a pattern. A desirable application of such a method would be a unified construction for all $s \geq 1$ for Case 11. Our current graph for $s = 1$, the first known current graph construction for finding a genus embedding of $K_{23}$, is a step towards that goal.

Some recent unifications were found by the author in the context of index 1 current graphs. Originally, these constructions were meant to improve Case 0-CG \cite{Sun-K12s} and Case 1-CG \cite{Sun-FaceDist}, but these current graphs also have arithmetic 3-ladders and hence also constitute unified constructions that improve upon those found in Jungerman and Ringel \cite{JungermanRingel-Minimal}. At present, Case 2 is the least unified of the residues. Triangular embeddings of $K_{12s+2}-K_2$ were found for all $s \geq 1$ by Jungerman~\cite{Jungerman-KnK2}, which by Construction~\ref{cons-mergehandle} yields genus embeddings of $K_{12s+2}$. The remaining minimum triangulations were found by an entirely different construction by Jungerman and Ringel~\cite{JungermanRingel-Minimal}. It seems plausible that lifting to index 3 current graphs may help, as it did with $K_{20}$ (see~\cite{Sun-FaceDist}) and $K_{23}$. 

\bibliographystyle{alpha}
\bibliography{../biblio}

\appendix

\newpage
\section{An alternate family of current graphs for $K_{12s+6}-K_3$}\label{app-case6}

In Figure~\ref{fig-current-c6-alt}, we give another index 3 construction for triangular embeddings of $K_{12s+6}-K_3$ using as much of the Bose ladder as possible. The corresponding segment of the Bose ladder has $4s-5$ rungs---if we had a family of current graphs where the varying portion was a Bose ladder with one more rung, then for $s=1$ an index 3 current graph would exist (with 0 rungs from the Bose ladder). Thus, we argue that this construction, combined with our experimental results showing nonexistence for $s=1$, is best possible for Case 6-CG. As a side note, the presence of Y-shaped fragments indicates that this family of current graphs does not utilize any more complicated building blocks than those known to Ringel~\emph{et al.}.

\begin{figure}[ht!]
\centering
\includegraphics[scale=0.85]{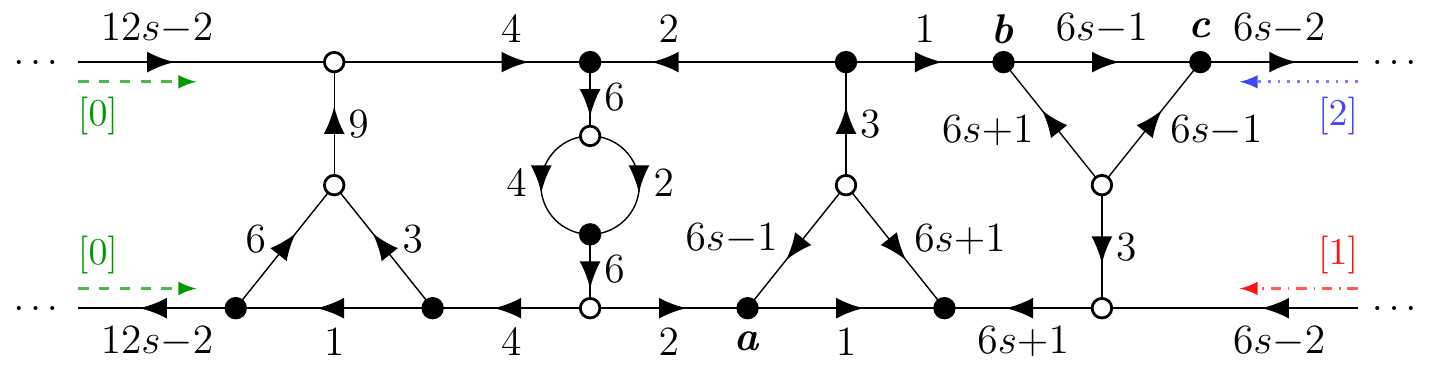}
\caption{Another construction for triangular embeddings of $K_{12s+6}-K_3$.}
\label{fig-current-c6-alt}
\end{figure}

\section{Small current graphs, Cases 9 and 11}\label{app-case911}

\subsection{Case 9}
For $s = 1$ we use the special current graph in Figure~\ref{fig-current-c9-s1}. It is essentially one of the inductive constructions used by Jungerman and Ringel~\cite{JungermanRingel-Minimal}, with the additional observation that the current graph used has an arithmetic 3-ladder.

\begin{figure}[ht!]
\centering
\includegraphics[scale=0.85]{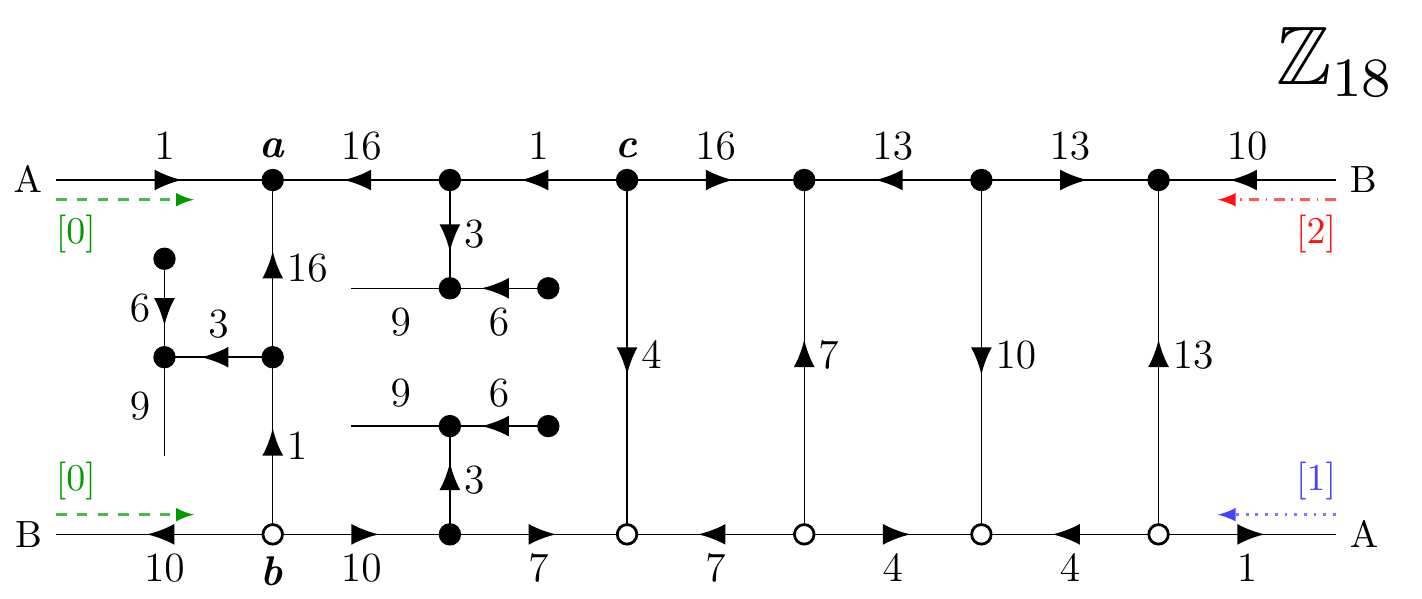}
\caption{A current graph for $K_{21}-K_3$ with an arithmetic 3-ladder.}
\label{fig-current-c9-s1}
\end{figure}

\subsection{Case 11}\label{subsec-app11}

For $s = 1,2$, we first find a current graph with group $\Z_{12s+6}$ that generates a triangular embedding of $K_{12s+11}-K_5$. For $s = 1$, consider the index 3 current graph in Figure~\ref{fig-current-c11-s1}. The rotations at vertices $1$ and $12$ are of the form
$$
\begin{array}{ccccccccccccccccccccccccccccccccccccccccccc}
1. & \dots & \vora & 3 & \vorb & 5 & \vorc & 9 & \vord & 8 & \vore & \dots
\end{array} 
$$ 
$$
\begin{array}{ccccccccccccccccccccccccccccccccccccccccccc}
12. & \dots & 5 & 8 & \dots & 3 & 9 & \dots, 
\end{array}
$$
so applying Lemma~\ref{lem-k5} with $u = 1, v = 12, (p_1,p_2,p_3,p_4) = (5, 8, 3, 9)$ yields $(23,10)$- and $(23,4)$-triangulations, and a genus embedding of $K_{23}$. For $s = 2$, the current graph in Figure~\ref{fig-current-c11-s2} generates a triangular embedding of $K_{35}-K_5$. Similar to the $s=1$ case, we use the rotations
$$
\begin{array}{ccccccccccccccccccccccccccccccccccccccccccc}
2. & \dots & \vora & 10 & \vorb & 6 & \vorc & 7 & \vord & 4 & \vore & \dots
\end{array} 
$$ 
$$
\begin{array}{ccccccccccccccccccccccccccccccccccccccccccc}
19. & \dots & 7 & 10 & \dots & 6 & 3 & \dots, 
\end{array}
$$
and Lemma~\ref{lem-k5} to find the $(35,10)$- and $(35,4)$ triangulations, and a genus embedding of $K_{35}$. The remaining minimum triangulations can be found using the arithmetic $3$-ladder. 

\begin{figure}[ht!]
\centering
\includegraphics[scale=0.85]{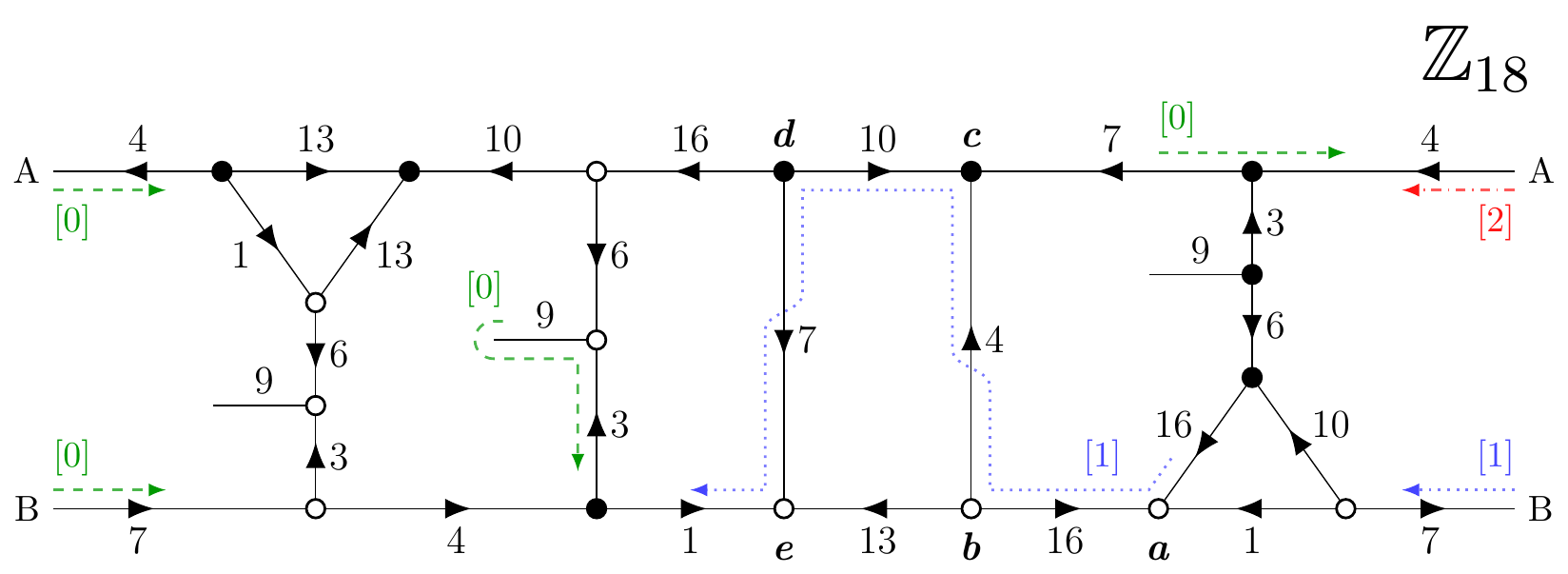}
\caption{An index 3 current graph for $K_{23}-K_5$.}
\label{fig-current-c11-s1}
\end{figure}

\begin{figure}[ht!]
\centering
\includegraphics[scale=0.85]{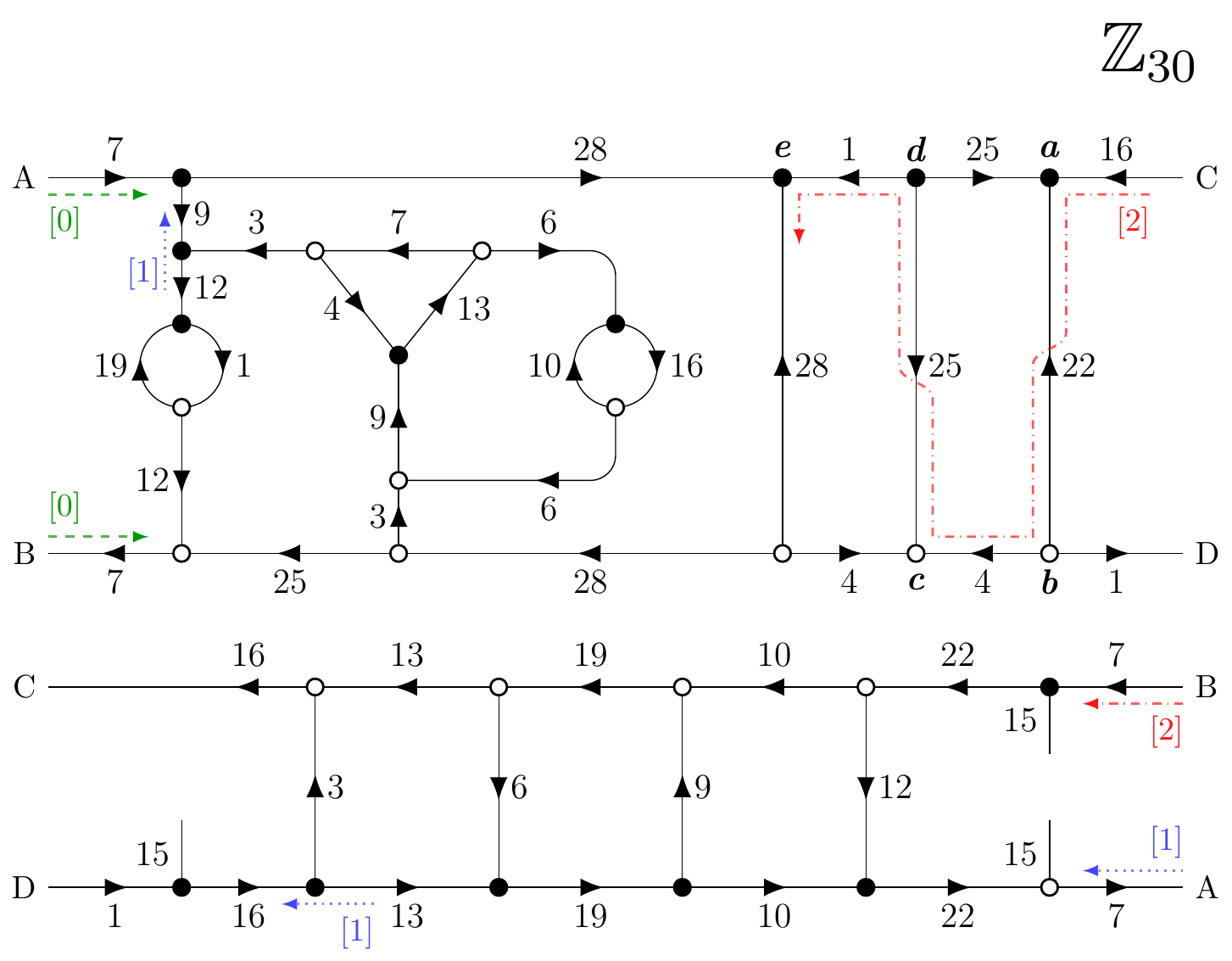}
\caption{An index 3 current graph for $K_{35}-K_5$.}
\label{fig-current-c11-s2}
\end{figure}

An embedding is said to be \emph{nearly triangular} if it has at most one face. The following result relates nearly triangular embeddings to minimum triangulations:

\begin{proposition}
If there exists a minimum genus embedding $\phi: K_n \to S_g$ of the complete graph $K_n$ with exactly one nontriangular, simple face, there exists a minimum triangulation of the surface $S_g$ on $n+1$ vertices.
\end{proposition}
\begin{proof}
The bounds derived from Heawood numbers (Propositions~\ref{prop-ry} and \ref{prop-jr}) give a lower bound of $n+1$ on the number of vertices in a minimum triangulation of $S_g$. Subdividing the nontriangular face with a new vertex yields the desired triangulation. 
\end{proof}

In particular, the aforementioned nonexistence result for $(9,3)$-triangulations due to Huneke~\cite{Huneke-Minimum} was used to show that $K_8$ does not have a nearly triangular embedding in $S_2$~\cite{Sun-FaceDist}. We use the nearly triangular genus embedding of $K_{22}$ given in \cite{Sun-FaceDist} to construct the remaining $(23,16)$-triangulation.

Finally, an \emph{ad hoc} unification of the $11$-vertex case is given in Appendix~\ref{app-sporadic}.

\section{Some small embeddings}\label{app-sporadic}

We collect a few special embeddings in this section. The first such embedding is of $K_8$, after deleting $q_0$ and $q_1$:

$$\begin{array}{rrrrrrrrrrrr}
0. & 2 & 7 & 3 & 1 & 4 & 5 & 6 & q_0 \\
2. & 4 & 1 & 5 & 3 & 6 & 7 & 0 & q_0 \\
4. & 6 & 3 & 7 & 5 & 0 & 1 & 2 & q_0 \\  
6. & 0 & 5 & 1 & 7 & 2 & 3 & 4 & q_0 \\
1. & 7 & 6 & 5 & 2 & 4 & 0 & 3 & q_1 \\
3. & 1 & 0 & 7 & 4 & 6 & 2 & 5 & q_1 \\
5. & 3 & 2 & 1 & 6 & 0 & 4 & 7 & q_1 \\
7. & 5 & 4 & 3 & 0 & 2 & 6 & 1 & q_1 \\
q_0. & 6 & 4 & 2 & 0 \\
q_1. & 1 & 3 & 5 & 7
\end{array}$$

This embedding was used in several ways: it is a minimum triangulation of $S_2$, it is a genus embedding of $K_9$ after amalgamating $q_0$ and $q_1$, and three of the handles of Lemma~\ref{lem-k8} can be thought of as gluing this embedding at two quadrilateral faces.

Known $(11,4)$-triangulations and genus embeddings of $K_{11}$ do not follow naturally from current graph constructions. To lessen the load of having to verify these special embeddings, we give a triangular embedding of $K_{11}-C_4$:
$$\begin{array}{rrrrrrrrrrrr}
0. & 1 & 10 & 8 & 4 & 2 & 9 & 7 & 5 & 3 & 6 \\
1. & 0 & 6 & 4 & 8 & 5 & 9 & 3 & 7 & 2 & 10 \\
2. & 0 & 4 & 10 & 1 & 7 & 6 & 5 & 8 & 3 & 9 \\
3. & 0 & 5 & 10 & 4 & 7 & 1 & 9 & 2 & 8 & 6 \\
4. & 0 & 8 & 1 & 6 & 9 & 5 & 7 & 3 & 10 & 2 \\
5. & 0 & 7 & 4 & 9 & 1 & 8 & 2 & 6 & 10 & 3 \\
6. & 0 & 3 & 8 & 10 & 5 & 2 & 7 & 9 & 4 & 1 \\
7. & 0 & 9 & 6 & 2 & 1 & 3 & 4 & 5 \\
8. & 0 & 10 & 6 & 3 & 2 & 5 & 1 & 4 \\
9. & 0 & 2 & 3 & 1 & 5 & 4 & 6 & 7 \\
10. & 0 & 1 & 2 & 4 & 3 & 5 & 6 & 8 \\
\end{array}$$
The missing edges are $(7,8)$, $(8,9)$, $(9,10)$, and $(10,7)$, which can be added with one handle using Construction~\ref{cons-mergehandle} as in Figure~\ref{fig-k11}. Note that this construction does not really make use of any specific structure in the embedding, as we can always find a face incident with a given edge. We thus formulate this additional adjacency approach more generally:

\begin{proposition}
If there exists a triangular embedding of $K_n-C_4$, then there exists a genus embedding of $K_n$. 
\end{proposition}

\begin{figure}[ht]
\centering
\includegraphics{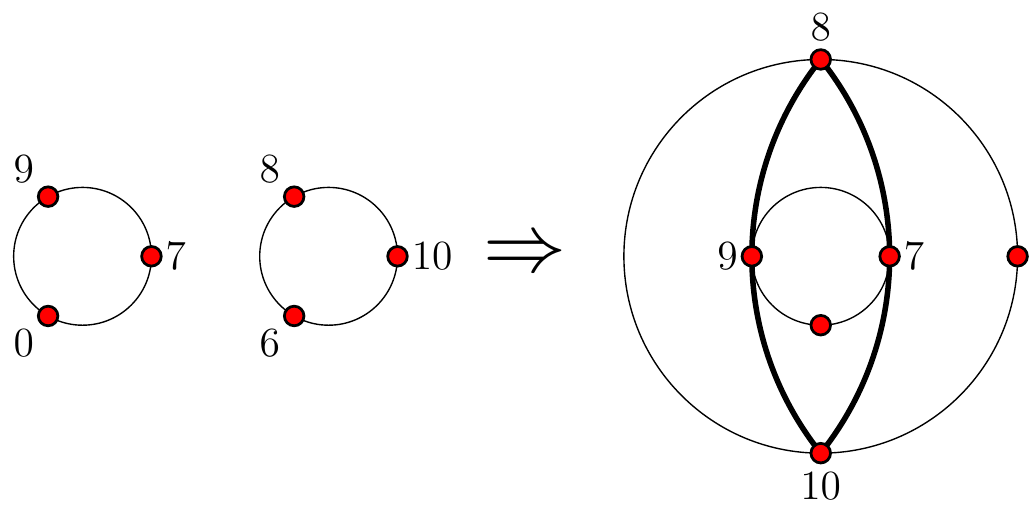}
\caption{A generic method for adding a $C_4$ with one handle, applied to the triangular embedding of $K_{11}-C_4$.}
\label{fig-k11}
\end{figure}

\end{document}

%% file: letter-defn.txt
\newcommand{\vora}{\textbf{\emph{a}}}
\newcommand{\vorb}{\textbf{\emph{b}}}
\newcommand{\vorc}{\textbf{\emph{c}}}
\newcommand{\vord}{\textbf{\emph{d}}}
\newcommand{\vore}{\textbf{\emph{e}}}
\newcommand{\vorf}{\textbf{\emph{f}}}
\newcommand{\vorg}{\textbf{\emph{g}}}
\newcommand{\vorh}{\textbf{\emph{h}}}

\newcommand{\leta}{\textbf{{a}}}
\newcommand{\letb}{\textbf{{b}}}
\newcommand{\letc}{\textbf{{c}}}
\newcommand{\letd}{\textbf{{d}}}
\newcommand{\lete}{\textbf{{e}}}
\newcommand{\letf}{\textbf{{f}}}
\newcommand{\letg}{\textbf{{g}}}
\newcommand{\leth}{\textbf{{h}}}

\tikzstyle{logblue}=[dotted, thick, blue!73!white, -latex, draw opacity=0.7]
\tikzstyle{loggreen}=[dashed, thick, green!60!black, -latex, draw opacity=0.7]
\tikzstyle{logred}=[dash dot, thick, red!92!white, -latex, draw opacity=0.7]

%% file: bose.bbl
\begin{thebibliography}{Sun19b}

\bibitem[Bou78]{Bouchet-Diamond}
Andr{\'e} Bouchet.
\newblock Orientable and nonorientable genus of the complete bipartite graph.
\newblock {\em Journal of Combinatorial Theory, Series B}, 24(1):24--33, 1978.

\bibitem[GG{\v{S}}98]{Grannell-SurfaceEmbeddings}
M.J. Grannell, T.S. Griggs, and Jozef {\v{S}}ir{\'a}{\v{n}}.
\newblock Surface embeddings of {S}teiner triple systems.
\newblock {\em Journal of Combinatorial Designs}, 6(5):325--336, 1998.

\bibitem[GR76]{GuyRingel}
Richard~K. Guy and Gerhard Ringel.
\newblock {Triangular imbedding of $K_n{-}K_6$}.
\newblock {\em Journal of Combinatorial Theory, Series B}, 21(2):140--145,
  1976.

\bibitem[GR{\v{S}}07]{Goddyn-Exponential}
Luis Goddyn, R.~Bruce Richter, and Jozef {\v{S}}ir{\'a}{\v{n}}.
\newblock Triangular embeddings of complete graphs from graceful labellings of
  paths.
\newblock {\em Journal of Combinatorial Theory, Series B}, 97(6):964--970,
  2007.

\bibitem[GT87]{GrossTucker}
Jonathan~L. Gross and Thomas~W. Tucker.
\newblock {\em {Topological Graph Theory}}.
\newblock John Wiley \& Sons, 1987.

\bibitem[Gus63]{Gustin}
William Gustin.
\newblock Orientable embedding of {C}ayley graphs.
\newblock {\em Bulletin of the American Mathematical Society}, 69(2):272--275,
  1963.

\bibitem[GY73]{GuyYoungs-Smooth}
Richard~K. Guy and J.W.T. Youngs.
\newblock A smooth and unified proof of cases 6, 5 and 3 of the {Ringel-Youngs
  Theorem}.
\newblock {\em Journal of Combinatorial Theory, Series B}, 15(1):1--11, 1973.

\bibitem[Hun78]{Huneke-Minimum}
John~Philip Huneke.
\newblock A minimum-vertex triangulation.
\newblock {\em Journal of Combinatorial Theory, Series B}, 24(3):258--266,
  1978.

\bibitem[JR80]{JungermanRingel-Minimal}
Mark Jungerman and Gerhard Ringel.
\newblock Minimal triangulations on orientable surfaces.
\newblock {\em Acta Mathematica}, 145(1):121--154, 1980.

\bibitem[Jun74]{Jungerman-K18}
Mark Jungerman.
\newblock {Orientable triangular embeddings of $K_{18}-K_3$ and $K_{13}-K_3$}.
\newblock {\em Journal of Combinatorial Theory, Series B}, 16(3):293--294,
  1974.

\bibitem[Jun75]{Jungerman-KnK2}
Mark Jungerman.
\newblock The genus of {$K_n-K_2$}.
\newblock {\em Journal of Combinatorial Theory, Series B}, 18(1):53--58, 1975.

\bibitem[KV02]{KorzhikVoss}
Vladimir~P. Korzhik and Heinz-J{\"u}rgen Voss.
\newblock Exponential families of non-isomorphic non-triangular orientable
  genus embeddings of complete graphs.
\newblock {\em Journal of Combinatorial Theory, Series B}, 86(1):186--211,
  2002.

\bibitem[May69]{Mayer-Orientables}
Jean Mayer.
\newblock Le probleme des r{\'e}gions voisines sur les surfaces closes
  orientables.
\newblock {\em Journal of Combinatorial Theory}, 6(2):177--195, 1969.

\bibitem[PJ79]{Pengelley-Index4}
David~J. Pengelley and M.~Jungerman.
\newblock Index four orientable embeddings and case zero of the {Heawood}
  conjecture.
\newblock {\em Journal of Combinatorial Theory, Series B}, 26(2):131--144,
  1979.

\bibitem[Rin61]{Ringel-1961}
Gerhard Ringel.
\newblock {\"U}ber das {P}roblem der {N}achbargebiete auf orientierbaren
  {F}l{\"a}chen.
\newblock In {\em Abhandlungen aus dem Mathematischen Seminar der
  Universit{\"a}t Hamburg}, volume~25, pages 105--127. Springer, 1961.

\bibitem[Rin74]{Ringel-MapColor}
Gerhard Ringel.
\newblock {\em {Map Color Theorem}}, volume 209.
\newblock Springer Science \& Business Media, 1974.

\bibitem[RY68]{RingelYoungs}
Gerhard Ringel and J.W.T. Youngs.
\newblock {Solution of the Heawood map-coloring problem}.
\newblock {\em Proceedings of the National Academy of Sciences},
  60(2):438--445, 1968.

\bibitem[RY69a]{RingelYoungs-Case11}
Gerhard Ringel and J.W.T. Youngs.
\newblock {Solution of the Heawood map-coloring problem --- Case 11}.
\newblock {\em Journal of Combinatorial Theory}, 7(1):71--93, 1969.

\bibitem[RY69b]{RingelYoungs-Case2}
Gerhard Ringel and J.W.T. Youngs.
\newblock {Solution of the Heawood map-coloring problem --- Case 2}.
\newblock {\em Journal of Combinatorial Theory}, 7(4):342--352, 1969.

\bibitem[Sun18]{Sun-FaceDist}
Timothy Sun.
\newblock Face distributions of embeddings of complete graphs.
\newblock {\em arXiv: 1708.02092}, 2018.

\bibitem[Sun19a]{Sun-Index2}
Timothy Sun.
\newblock Jungerman ladders and index 2 constructions for genus embeddings of
  dense regular graphs.
\newblock {\em in preparation}, 2019.

\bibitem[Sun19b]{Sun-K12s}
Timothy Sun.
\newblock A simple construction for orientable triangular embeddings of the
  complete graphs on $12s$ vertices.
\newblock {\em Discrete Mathematics}, 342(4):1147--1151, 2019.

\bibitem[You70]{Youngs-3569}
J.W.T. Youngs.
\newblock {Solution of the Heawood map-coloring problem --- Cases 3, 5, 6, and
  9}.
\newblock {\em Journal of Combinatorial Theory}, 8(2):175--219, 1970.

\end{thebibliography}
